\documentclass[11pt,reqno]{article}

\title{Renewal theory for random walks on surface groups}
\author{Peter Ha\"{\i}ssinsky, Pierre Mathieu, and  Sebastian M\"uller%\let\thefootnote\relax\footnote{This work was supported by the ESF grant PIEF-2009-235688.}
}

\usepackage{amsmath,amsthm,amsfonts,amssymb}
\usepackage{t1enc}
 \usepackage[ps2pdf]{hyperref}
\usepackage{graphicx}
 \usepackage{pstricks}
 \usepackage{epsfig}
 \usepackage{pst-grad} % For gradients
 \usepackage{pst-plot} % For axes

\newtheorem{thm}{Theorem}[section]

\newtheorem{cor}[thm]{Corollary}
\newtheorem{lem}[thm]{Lemma}

\theoremstyle{definition}
\newtheorem{defn}{Definition}[section]
\newtheorem{assumption}{Assumption}%[section]

\theoremstyle{remark}
\newtheorem{rem}{Remark}[section]
\newtheorem{que}{Question}[section]

\newtheorem{ex}{Example}[section]

\newcommand{\eps}{\varepsilon}
\newcommand{\g}{\gamma}
\newcommand{\A}{\mathcal{A}}

\newcommand{\GG}{\mathcal{G}}

\newcommand{\normal}{\mathcal{N}}

\newcommand{\bT}{\textbf{T}}

\newcommand{\x}{\textbf{x}}
\newcommand{\y}{\textbf{y}}
\newcommand{\e}{\textbf{e}}

\newcommand{\T}{\mathcal{T}}
\newcommand{\D}{\mathcal{D}}

\newcommand{\R}{\mathbb{R}}
\newcommand{\C}{\mathbb{C}}
\newcommand{\Z}{\mathbb{Z}}
\newcommand{\N}{\mathbb{N}}
\newcommand{\1}{\textbf{1}}
\renewcommand{\i}{\textbf{i}}

\renewcommand{\P}{\mathbb{P}}
\newcommand{\Q}{\mathbb{Q}}

\newcommand{\Paths}{\mathcal{P}}

\newcommand{\E}{\mathbb{E}}

\renewcommand{\O}{\mathcal{O}}

\newcommand{\eval}[2][\right]{\relax
  \ifx#1\right\relax \left.\fi#2#1\rvert}

\begin{document}
\maketitle {\abstract 
We construct a renewal structure for random walks on surface groups. 
The renewal times are defined as times when the random walks enters a particular type of a cone and never leaves it again. As a consequence, the trajectory of the random walk  can be expressed as an \emph{aligned union} of i.i.d.~trajectories between the renewal times. Once having established this renewal structure, we prove  a central limit theorem for the distance to the origin under exponential moment conditions.  Analyticity of the speed and of the asymptotic variance are natural consequences of our approach. Furthermore, our method applies to  groups with infinitely many ends and therefore generalizes classic results on central limit theorems on free groups.}

\section{Introduction} 
The idea and motivation behind a renewal theory for random walks on groups is to find a  decomposition of the  trajectory of the  walk into aligned pieces in such a way that these pieces are identically and independently distributed. 
The main result of this paper is the construction of such a renewal theory for random walks on an important class of hyperbolic groups. 

Analogous renewal structures have been developed for random walks on free groups and trees with finitely many cone types, Nagnibeda and Woess \cite{NW}, free products of groups  and regular languages, Gilch  \cite{G:07, G:08}. It is also a common technique used in the study of random walks in random environment in order to prove  laws of large numbers and  {annealed} central limit theorems, \emph{e.g.,} see \cite{Z}.  
However, to the best of our knowledge the renewal structure given in this paper constitutes the first example on one-ended groups beyond $\Z^{d}$.

We invite the reader to  consider the following simple but instructing example: nearest neighbor random walk  on the free group $\mathbb{F}_{2}=\langle a,b\rangle$ with two generators.   In order to define a suitable renewal structure  we recall the definition of cone types after Cannon, \emph{e.g.}, see \cite{ECHLPT}. A cone $C(x)$ consists of all vertices $y$ such that $x$ lies on the geodesic from the group identity $e$ to $y$.  A cone type describes the way one \emph{can look to infinity}, \emph{i.e.}, $T(x)=\{y:~xy\in C(x)\}$. It turns out that there are five different cone types, say $\textbf{e},\textbf{a},\textbf{a}^{-},\textbf{b},$ and $\textbf{b}^{-}$. The cone type $\textbf{e}$ corresponds to the cone of $e$ and $\textbf{x}$ corresponds to the one where an $x$ edge \emph{leads back} to the identity.   
The cone types $\textbf{a},\textbf{a}^{-},\textbf{b},$ and $\textbf{b}^{-}$ have the important property that a cone of one type contains cones of the other three types.
Furthermore, any irreducible random walk is transient and hence has a positive probability to stay in a cone $C(x)$ for all times. So,
 let us fix a cone type, say $\textbf{a}$, and  define the renewal  times $(R_{n})_{n\geq 1}$ as follows. Let $R_{1}$ be the first time that the walk visits a cone of type $\textbf{a}$ that it will never leave again. Inductively, we define $R_{n+1}$ as the first time after $R_{n}$ that the walk visits a cone of type $\textbf{a}$ that it will never leave again. One can check that $(R_{n+1}-R_{n})_{n\geq 1}$ is indeed an i.i.d.~sequence of random variables. Furthermore, we have (using the definition of the cones) that 
$$d(Z_{R_{n}},e)=d(Z_{R_{1}},e)+\sum_{i=1}^{n-1}d(Z_{R_{i+1}}, Z_{R_{i}}),$$ where $d(\cdot,\cdot)$ is the word metric.
Non-amenability of the free group implies that  the random variables in the above equation all have some exponential moments. It is now standard to deduce a law of large numbers and a central limit theorem for the distance to the origin. Moreover, the renewal structure enables us to describe the rate of escape and the asymptotic variance, in terms of first and second moments of random variables with exponential moments. This fact allows a very good control of the regularity of these two quantities. 

The main technical difficulties that arise when developing the above sketch into a mathematical proof  are due to the fact that  the random times $R_{n}$ are not stopping times but depend on future events of the walk. Hence, conditioning on events described by $R_{n}$ destroys the Markovian structure of the random walk. 
Furthermore, for general hyperbolic groups it is not known wether  the Cannon automaton gives rise to cones with as nice properties as the ones in the free group. However, we shall prove that one can get a nice  renewal structure for random walks on surface groups.

The next Subsection  contains a short introduction on central limit theorems for random walks on groups and Subsection \ref{intr:anal} concerns recent results on analyticity of the rate of escape.
Section \ref{sec:not} prepares the ground for the main results in giving the necessary notation, definitions and preliminary results. In Section \ref{sec:ren} the renewal structure is formally defined and the main results are proven.

\subsection{Central limit theorem}\label{intr:clt}

Let $(X_{i})_{i\geq 0}$ be i.i.d~random variables taking values in $\Z^{d}$. Under a second moment condition we have the classical central limit theorem (CLT)  
$$\frac{\sum_{i=1}^{n} X_{i}-nv}{\sqrt n } \xrightarrow[n \to \infty]{\mathcal{D}} \mathcal{N}(0,\sigma^{2}),$$  
where $v=\E[X_{1}]$ is the rate of escape (or drift) and $\sigma^{2}$ the asymptotic variance.
A natural question, which goes back to Bellman \cite{bellman} and Furstenberg and Kesten \cite{FK}, is to which extent this phenomenon generalizes to $(X_{i})_{i\geq 1}$ taking values in some finitely generated group $\Gamma$.  Let $d(\cdot,\cdot)$ be a  left invariant metric on $\Gamma$ and $\mu$ a probability measure  whose support generates the group $\Gamma$. Let $(X_{i})_{i\geq 1}$ be i.i.d.~random variables with distribution $\mu$ and define the random walk  $Z_{n}= X_{1} X_{2} \cdots  X_{n} .$  Then, if $d(X_{1},e)$ has a finite first moment, Kingman's subadditive ergodic theorem ensures that 
$$\lim_{n\to\infty} \frac1n d(Z_{n},e)=:v$$ exists  in the almost sure and $L^1$ senses and is deterministic. 
In other words, there is a law of large numbers for random walks on groups. 
Moreover, Guivarc'h \cite{Gui:80} proved that if $\Gamma$ is a non-amenable finitely generated group then any random walk with a finite first moment has positive rate of escape with respect to any word metric.
However, it turns out that a central limit theorem can not be stated in this general setting. As described in Bj\"orklund \cite{bjorklund} one can use the result of Erschler (\cite{E:99}, \cite{E:01}) to construct the following counterexample. Let  $\Gamma= (\Z \wr \Z) \times \mathbb{F}_{2}$ where $\wr$ is the wreath product and  $\mathbb{F}_{2}$ denotes the free group on two generators. There exists a symmetric probability measure $\mu$ with finite support on $\Gamma$ and a word metric $d$ such that the fluctuations around the linear (positive) drift are of order $n^{\frac34}$. 

However, there are several situations where central limit theorems are established. Sawyer and Steger in \cite{ST} studied the case of the free group $\mathbb{F}_{d}$ with $d$ standard generators and the  corresponding word distance.  Under technical moment conditions they prove that $(d(Z_{n},e)-nv)/\sqrt{n}$ converges in law to some non-degenerated Gaussian distribution. While their proof uses analytic extensions of Green functions, another proof was given by Lalley \cite{La:93} using algebraic function theory and Perron--Frobenius theory. A geometric proof was later presented by Ledrappier \cite{Le:01}. A generalization for trees with finitely many cone types can be found in Nagnibeda and Woess \cite{NW}.  Recently, Bj\"orklund \cite{bjorklund} proved a central limit theorem on hyperbolic groups with respect to the Green metric. The proof in \cite{bjorklund} is based on the identification of the Gromov boundary with the horofunction boundary. This fact enables to prove the CLT using  a martingale approximation. However, the CLT for the Green metric does not seem to imply directly the central limit theorem for the drift with respect to any word metric on $\Gamma$. 

One of the main objectives of this paper is to demonstrate that a CLT for one-ended groups, Theorem \ref{thm:CLT_planar}, can also be obtained by using a renewal structure. We have managed to do so for random walks on surface groups. However, we believe that our approach  should work in the general setting of hyperbolic groups, see Section \ref{sec:discussion} for a short discussion. 

\begin{thm}\label{thm:CLT_planar}
Let $\Gamma$ be a surface group with standard generating set $S$ and corresponding  word metric $d$. Furthermore, let 
 $\mu$ be a driving measure with exponential moments  whose support contains the generating set $S$. Then
$$ \frac{d(Z_{n}, e) -nv }{\sqrt{n}} \xrightarrow[n \to \infty]{\mathcal{D}}  \normal (0, \sigma^2),$$
with \begin{equation*}v=\frac{\E[ d(Z_{ R_{2}}, Z_{ R_{1}})]}{\E[  R_{2}- R_{1}]} \mbox{ and }
\sigma^{2}=\frac{\E[(d(Z_{R_{2}}, Z_{R_{1}}) - (R_{2}- R_{1})v)^{2}]}{\E[R_{2}-R_{1}]}.
\end{equation*}
\end{thm} 

\begin{rem}
An analogous result holds true for random walks on groups with infinitely many ends. In fact,  by applying Stalling's splitting theorem, 
% and 
one can check that all arguments  work fine for amalgamated free products and HNN-extensions over finite groups with a suitable choice of generators, see also Remark \ref{rem:inftyends}. 
\end{rem}

Let us note here that other results in this direction are known for actions of linear semigroups on projective spaces by Le Page \cite{LePage} and Guivarc'h and le Page \cite{GLeP:04}. We also want to mention earlier works of Tatubalin \cite{Tutubalin:65} and \cite{Tutubalin:68} on random walks on hyperbolic space. 

More recently, Pollicott and Sharp \cite{PS:10} use a thermodynamical formalism to prove  limit theorems for matrix groups acting cocompactly on the hyperbolic group. Calegari gives generalizations  to actions on general hyperbolic groups in the survey paper  \cite[Section 3]{Calegari}.

\subsection{Analyticity of the rate of escape and asymptotic variance}\label{intr:anal}
Fix a group $\Gamma$, a finite generating set $S$, and a probability measure $\mu$ on $\Gamma$. Let $v_{\mu}$ be the drift corresponding to $\mu$ with respect to the word metric induced by the generating set. A natural question asks whether $v_{\mu}$ (and the asymptotic entropy) depends continuously on $\mu$.   Continuity of the rate of escape (and the asymptotic entropy) is known on  hyperbolic groups under the more general condition of having a finite first moment, see  Kaimanovich and {\`E}rshler \cite{EK}. Analyticity of the rate of escape and of the asymptotic entropy on free groups was  proven by Ledrappier in \cite{Le:10}. More recently, Ledrappier \cite{Le:11} proves Lipschitz continuity for the rate of escape and asymptotic entropy for random walks on Gromov hyperbolic groups. 

Moreover, analyticity of the rate of escape also follows in certain cases where explicit formul\ae~for the rate of escape are known, see Mairesse and  Mathéus \cite{MM:07b}  and Gilch, \cite{G:07} and \cite{G:08}.  Mairesse and Mathéus \cite{MM:07} show that the rate of escape for some random walks on the Braid group $B_{3}=\langle a,b | aba = bab\rangle$ is continuous but not(!)~differentiable. Finally, we refer to the recent survey of Gilch and Ledrappier  \cite{GL} on results on the regularity of drift and entropy of random walks on groups.

The  central limit theorem, Theorem \ref{thm:CLT_planar}, provides formul\ae~for the drift $v$ and asymptotic variance $\sigma^{2}$ in terms of renewal times and hence offers a new approach in order to study analyticity. Moreover, this approach allows to consider random walks with infinite support.

Fix  a driving measure  $\nu$ of a random walk with exponential moments, \emph{i.e.}, $\E[\exp(\lambda d(X_{1},e))]<\infty$  for some $\lambda>0$. Let $B$ be a finite subset of the support of $\nu$, \emph{i.e.,}  $B \subseteq supp(\nu)$.
 Let $\Omega_{\nu} (B)$ be the set of probability measures that give positive weight to all elements of $B$ and coincide with $\nu$ outside $B$. The set  $\Omega_{\nu} (B)$ can be identified with an open bounded convex subset in $\R^{|B|-1}$. For each $\mu\in \Omega_{\nu}(B)$ we define the functions $v_{\mu}$ and $\sigma_{\mu}$ as the rate of escape and the asymptotic variance for the random walk with law $\mu$.
\begin{thm}\label{thm:analytic}
Let $\Gamma$ be a surface group with standard generating set $S$ and let $\nu$ be a driving measure with exponential moments whose support contains the generating set $S$. Then, for all $B$ such that $ B \subseteq supp(\nu)$ the functions $\mu\mapsto v_{\mu}$ and $\mu\mapsto \sigma_{\mu}$ are real analytic on $\Omega_{\nu}(B)$.
\end{thm}

\begin{rem}
After  this  paper was made publicly available in $2013$  considerable progress has been made. In particular, Benoist and Quint \cite{BQ} proved a central limit theorem for random walks on hyperbolic groups under the optimal second moment condition and Gou{\"e}zel \cite{Go:15} shows analyticity for the rate of escape, asymptotic variance, and asymptotic entropy for random walks with finite support. The approaches used in these works are based on spectral techniques and may be considered  ``antipodal'' to ours.
\end{rem}

\section{Notation and Preliminaries}\label{sec:not}

\subsection{Cone types, geodesic automata of hyperbolic groups}

Let $\Gamma$ be a finitely generated group and let $S$ be a symmetric and finite generating set. For sake of brevity we speak just of the group $(\Gamma,S)$ instead of the group $\Gamma$ together with a finite generating set $S$. The Cayley graph $X$ associated with $S$ is the graph whose  vertex set is the set of all group elements and whose edge set consists of all pairs $(\g,\g')\in\Gamma\times\Gamma$ such that $\g^{-1}\g'\in S$. Endowing $X$ with the length
metric  which makes each edge isometric to the segment $[0,1]$ defines the {\it word metric $d(\cdot,\cdot)$  associated with $S$}.  
This metric turns $X$ into a geodesic proper metric space on which $\Gamma$ acts geometrically by left-translation.

Let $B_{n}(x)$ be the ball of radius $n$ around $x$; set for brevity  $B_{n}=B_{n}(e)$. 
The neighborhood relation is written as $\sim$, \emph{i.e.}, $x\sim y$ if $d(x,y)=1$, or equivalently $x^{-1}y\in S$.
 A path is a sequence of adjacent vertices in $X$ and is denoted by  $\langle \cdot \rangle$.
 Let  $\A=(V_{\A}, E_{\A}, s_*)$ be a finite directed graph with vertex set $V_{\A}$, edge set $\E_{\A}$, and a  distinguished vertex $s_*$ together with
a labeling $\alpha: E_{\A} \to S$ of the edges. Vertices and edges of $\A$ will be denoted using bold fonts. % $\x$ and $\y$.  
The vertex set $V_{\A}$ will often be identified with $\A$, \emph{i.e.}, $\x\in \A$ means a vertex $\x\in V_{\A}$.

Denote the set  
$$\Paths:=\{ \mbox{finite paths in } \A \mbox{ starting in } s_*\}.$$ 
For $m\in\N\cup\{\infty\}$, each path $\gamma=\langle \x_{1},\ldots, \x_{m}\rangle \in \Paths$ 
gives rise to a path in $\Gamma$ starting from $e$.  Denote by $\e_{1}$ the edge between $s_{*}$ and $\x_{1}$ and by  $\e_{i}$ the edge between $\x_{i-1}$ and $\x_{i}$ for $i\geq 2$.  The path  corresponding to $\gamma$ is then defined by 
$$ \alpha(\gamma) =\langle e, \alpha(\e_{1}),\alpha(\e_{1})\alpha(\e_{2}),\ldots, \prod_{i=1}^{m} \alpha(\e_{i})\rangle.$$
 
\begin{defn}
An \emph{automatic structure} for a group $(\Gamma,S)$ is given by a finite state automaton $\A$ and a labeling $\alpha$ which
satisfy the following properties:
\begin{itemize}
\item no edge in $E_{\A}$ ends at $s_*$,
\item every vertex $v\in V_{\A}$ is accessible from $s_*$,
\item for every path $\gamma\in \Paths$, the path $\alpha(\gamma)$ is a geodesic path in $\Gamma$,
\item the mapping $\alpha^{*}$ from  $\Paths$ to $\Gamma$ which associates the endpoint of the geodesic is surjective.
\end{itemize}
We speak of a \emph{strongly automatic structure} if $\alpha^{*}$ defines a bijection between $\Paths$ and $\Gamma$.
\end{defn}

For hyperbolic groups, the existence of a (strongly) automatic structure is due to Cannon and based on the definition of cones. The  \emph{cone} after Cannon  of a group element $x$ is defined (for any choice of generating set) as
$$ C(x):=\{ y\in \Gamma:~d(e,y)=d(e,x)+d(x,y)\}.$$ We say $x$ is the \emph{root} of $C(x)$. The \emph{cone type} is defined as
$$ T(x):=\{ y\in \Gamma:~d(e,xy)=d(e,x)+d(x,xy)\} = x^{-1}C(x).$$

Cannon's fundamental result, see \emph{e.g.,}  \cite{ECHLPT}, is that a hyperbolic group 
has only finitely many cone types. Furthermore, we may thus associate a directed graph  $\A_C=(V_{\A}, E_{\A}, s_*)$ with distinguished vertex $s_*$ together with
a labeling $\alpha: E_{\A} \to S$ of the edges as follows.
The set of vertices $V_{\A}$ is the set of cone types, and $s_*=T(e)$ is the cone type of the neutral element;
there is a directed edge $\e=(\bT_1,\bT_2)$ labeled by $s$ ($\alpha(\e)=s$) between two cone types if there is an element $x\in\Gamma$ 
such that  $T(x)= \bT_1$, $T(xs)=\bT_2$ and $s\in T(x)$.
This structure $(\A_C,\alpha)$ is by definition the {\it Cannon automaton} of $(\Gamma,S)$.
We may obtain a strongly automatic structure from $\A_C$ by choosing a lexicographic ordering of the cone types; see \cite{ECHLPT} for details. 

Furthermore, any  strongly  automatic structure $\A$ defines \emph{cones} $C_{\mathcal{A}}$ and \emph{cone types} $T_{{\mathcal{A}}}$ as follows.
Given $x\in\Gamma$ and a path $\gamma_x\subset\A$ representing $x$, we let $C_\mathcal{A}(x)$ denote the set of all points of $\Gamma$
which are represented by paths with $\gamma_x$ as prefix. The {cone} $C_\mathcal{A}(x)$ is well defined since its construction does not depend on the choice of the representing path $\gamma_{x}$. We say that $x$ is the \emph{root} of the cone $C_\mathcal{A}(x)$. Moreover, we can define {cone types} as $T_\mathcal{A}(x)=x^{-1}C_\mathcal{A}(x)$ which has the neutral element $e$ as root.

A vertex $\y$ (or cone type) is  \emph{accessible} from $\x$ if there is a path from $\x$ to $\y$. 
In this case we write $\x\to\y$.  A vertex $\x\in \A$ is  \emph{recurrent} if $\x\to\x,$ otherwise 
it is called \emph{transient}. The set of recurrent vertices $\mathcal{R}$ induces a subgraph $\A_{\mathcal{R}}$  of $\A$, \emph{i.e.,} the graph whose vertex set equals to $\mathcal{R}$ and two vertices $\x$ and $\y$ are joint by an edge if only if they are neighbors in $\A$.  By extension and abuse of standard notation, we will say $x\in\Gamma$ is \emph{recurrent} if its cone type is recurrent. Recall that a (directed) graph is \emph{strongly connected} if every vertex is reachable from any other vertex by following the directions.

\begin{assumption}\label{ass:1}
There exists an automatic structure $\A$ associated to $S$ such that the subgraph $\A_{\mathcal{R}}$ is strongly connected. 
\end{assumption}

This assumption is verified for non-exceptional Fuchsian groups with  particular generating sets as shown is \cite{Se:82}.
In particular, surface groups with  standard generating sets  satisfy Assumption \ref{ass:1}. Not astonishingly, it also holds for groups with infinitely many ends for a suitable choice of generators. 
In fact this is a  consequence of Stalling's splitting theorem; any finitely generated group $\Gamma$ has more than one end if and only if 
the group splits as an amalgamated free product or an HNN-extension  over a finite subgroup of $\Gamma$.

In the sequel we need the following definitions.
Let us say a cone type $\bT$ is {\it large} if it is a neighborhood in $\Gamma\cup \partial\Gamma$ 
of a boundary point of the Gromov hyperbolic boundary $\partial \Gamma$. Any cone type containing a large cone type is again large. Moreover, we have the following fact.

\begin{lem}\label{lem:intnonempty}
Let $(\Gamma,S)$ be a non-elementary hyperbolic group and $\mathcal{A}$ an automatic structure satisfying Assumption \ref{ass:1}. If there exists at least one large recurrent cone type,   then all recurrent cone types are large.
\end{lem}

We shall say that a cone type $\bT$ is {\it ubiquitous} if there exists some $R$ such that any ball $B_R(x)$ in $X$ contains a 
vertex $y$ with $T(y)=\bT$. A ubiquitous cone type is recurrent, and under Assumption \ref{ass:1} 
every recurrent cone type is ubiquitous. 
We define the (inner) boundary of a cone $C_\mathcal{A}(x)$ as 
$$\partial_{\Gamma} C_\mathcal{A}(x):=\{y\in  C_\mathcal{A}(x):~ \exists z\in \Gamma\setminus C_\mathcal{A}(x)~\mbox{such that}~z\sim y\}$$ and   $\partial_{\infty} C_\mathcal{A}(x)$ 
as the closure of $C_\mathcal{A}(x)$ at infinity, \emph{i.e.}, in the Gromov hyperbolic compactification.
Let $\gamma=\langle x_{1}, x_{2}, \ldots\rangle $ be a geodesic, we also denote by $\gamma$ the set $\{ x_{1}, x_{2}, \ldots\}$.

\subsection{Use of constants} 
Constants in capital letters are chosen \emph{sufficiently large} and constants in small letters stand for positive constants that are \emph{sufficiently small.}
Constants without any label, \emph{e.g.,} $C$, are considered to be \emph{local}, \emph{i.e.}, their values may change from line to line.  Labelled constants, \emph{e.g.,}~$C_{h}$, are defined \emph{globally} and their values do not change as the paper goes along.

\subsection{Random walks on groups}
Let $\Gamma$ be a finitely generated group 
%is denoted by $\Gamma$ and 
and $S$ a  symmetric and finite generating set.  
Let $\mu$ be a probability measure on $\Gamma$ with support generating $\Gamma$ as a semigroup. By definition, the random walk associated with $\mu$ is the Markov chain with state space $\Gamma$ 
and transition probabilities $p(x,y)=\mu(x^{-1}y)$ for $x,y\in \Gamma$. The measure $\mu$ is called the driving measure of the random walk. 
We shall use the notation $\T= \Gamma^{\mathbb{N}}$ for the path space and $Z_n$ for the position of the walk at time $n$ and 
$X_n:=Z_{n-1}^{-1}Z_n$ for its increment. 

Let $\P_{x}$ denote the  distribution  of the random walk $(Z_{n})_{n\geq 0}$ when started at $x\in\Gamma$, and write $\P$ for $\P_{e}$. 
Observe that $\P_x$ is also the unique probability measure on $\T$ 
under which $Z_0=x$ and the $X_n$'s are i.i.d.~random variables with law $\mu$. 
On the set of trajectories $\T$ we will also make use of the shift map $\theta:\T\to\T$ defined by $\theta[ (z_n)_{n\ge 0}] = (z_{n+1})_{n\ge 0}$.

An elementary hyperbolic group is either finite or has two ends. Random walks on non-elementary hyperbolic groups are transient. 
As soon as the law $\mu$ has a finite first moment, \emph{i.e.,} $\E[d(e,X_{1})]<\infty$, 
the random walk $Z_{n}$ converges $\P$-a.s.~to some point $Z_{\infty}$ in the Gromov hyperbolic boundary $\partial\Gamma$, see Theorem 7.3 in \cite{Kai:00}.

The harmonic measure $\nu$ is defined as the law of $Z_{\infty}$. In other words, it is the probability measure 
on $\partial \Gamma$ 
such that $\nu(A)=\P[Z_{\infty}\in A]$ for $A\subset\partial\Gamma$.
Since $\Gamma$ is non-elementary and the random walk is assumed to be irreducible we have that $\nu(\xi)=0$ 
for all $\xi\in\partial \Gamma$ and $\nu(O)>0$ for any open set $O\subset \partial \Gamma$.
 
\begin{lem}\label{lem:stayincone}
Let $(\Gamma, S)$ be a non-elementary hyperbolic group and $\mathcal{A}$ a corresponding automatic structure. 
Let $\mu$ be a driving measure whose support generates $\Gamma$ as a semigroup.
If $\bT$ is a large cone type in $\mathcal{A}$ then for all $x$ with $T_\mathcal{A}(x)=\bT$ we have that $$\P_{x}[Z_{n}\in C_\mathcal{A}(x) \mbox{ for all but finitely many } n\geq 0]>0.$$ 
\end{lem}
\begin{proof}
Let $O$ be an open subset of $\partial_{\infty}C_\mathcal{A}(x)$.
On the event that $Z_{\infty}\in O$, at some moment, 
the random walk $(Z_n)$ enters  $C_\mathcal{A}(x)$ and never leaves it afterwards.   
\end{proof}

Recall that our aim is to define a sequence of renewal times that corresponds to a sequence of cones in which the random walks stays forever. 
Therefore, we need the statement of Lemma \ref{lem:stayincone} to hold for all $n\in\N$. The next assumption is made to ensure this; however we see in Section \ref{sec:bypass} how to bypass this assumption.

\begin{assumption}\label{ass:2}
The support of the driving measure $\mu$ contains the generating set $S$  of the group $\Gamma$.
\end{assumption}

\begin{lem}\label{lem:stayinconegen}
Let $(\Gamma, S)$ be a non-elementary hyperbolic group  and $\mathcal{A}$ a corresponding automatic structure. Let $\mu$ be a driving measure that satisfies Assumption \ref{ass:2}. Then, there exists some $c>0$ such that 
for all $x$ such that $T(x)$ is large  we have that $$\P_{x}[Z_{n}\in C_\mathcal{A}(x) \mbox{ for all  } n\geq 0]>c.$$ 
\end{lem}

\begin{proof}
The event $\{Z_{0}=x, Z_{n}\in C_\mathcal{A}(x) \mbox{ for all  } n\geq 0\}$ consists only of  trajectories that stay inside the cone $C_\mathcal{A}(x)$. Hence, invariance of the walk implies that for $x,y$ such that $T_\mathcal{A}(x)=T_\mathcal{A}(y)$ we have
\[
\P_{x}[Z_{n}\in C_\mathcal{A}(x) \mbox{ for all  } n\geq 0]=\P_{y}[Z_{n}\in C_\mathcal{A}(y) \mbox{ for all  } n\geq 0].
\]
Since there is only a finite number of cone types it suffices to prove that the latter probability is positive for vertices of large cone types. Let $x$ be such that $T(x)$ is large. Now,  Lemma \ref{lem:stayincone}   implies that there exists some $y\in C_\mathcal{A}(x)$ such that
$\P_{y}[Z_{n}\in C_\mathcal{A}(x) \mbox{ for all } n\geq 0]>0.$ Assumption \ref{ass:2} guarantees that there exists  $n$  such that $\P_{x}[Z_{n}=y,~Z_{k}\in C_\mathcal{A}(x)~\forall k\leq n]>0$. The claim now follows by applying the law of total probability and the  Markov property of the random walk. 
\end{proof}

We will need the following result on the exponential decay of transition probabilities. We do not claim to be the first observing the result in this form but are not aware of  any reference. It is essentially a reformulation of \cite{BC:74} and \cite{DeGu:73}.

\begin{lem}\label{lem:expdecay}
Let $\Gamma$ be an infinite, finitely generated non-amenable group and let $\mu$ be a measure driving the random walk $(Z_{n})_{n\geq 0}$. Then there exists a constant $\varrho<1$ such that for all $x,y\in \Gamma$ and all $n\geq 1$,
\begin{equation*}\label{eq:exponentialdecay}
\P[Z_{n}=x]\leq \varrho^{n}.
\end{equation*}
\end{lem}
\begin{proof}
We consider $P$ as the convolution operator on $\ell_{2}(\Gamma)$ defined by
\begin{equation*}
Pf(x)=\sum_{y} \mu(x^{-1}y) f(y).
\end{equation*}
In the case where $\mu(e)>0$ it is proven in \cite{BC:74} that the operator norm $r:=\| P\|$ of $P$  is strictly less than $1$.  We also refer to \cite{DeGu:73} for related results on the spectral radius of $P$. Hence, 
\begin{equation*}
\sum_{y} (P^{n}f(y))^{2}\leq r^{n} \sum_{y}(f(y))^{2},\quad n\geq 1.
\end{equation*}
Choosing for $f$ the indicator $\delta_{x}$ of a point $x\in \Gamma$ gives
\begin{equation*}
\P[Z_{n}=x]=\mu^{n}(x)\leq (\sqrt{r})^{n}.
\end{equation*}
This finishes the proof with $\varrho=\sqrt{r}$ for the case $\mu(e)>0$.
Let us now treat the general case. Let $\bar \mu:=\frac12 \delta_{e}+\frac12 \mu$ and apply the result from above to $\bar \mu$:
\begin{equation*}
\bar \mu^{n}(x) \leq (\sqrt{\bar{r}})^{n},
\end{equation*}
for some $\bar r<1$. Since $\bar \mu^{n} =\sum_{k=0}^{n} 2^{-n} {n \choose k} \mu^{k}$ we have $\mu^{n}\leq 2^{n} ({2n\choose n})^{-1} (\bar \mu)^{2n}$ and thus 
\begin{equation*}
\mu^{n}(x)\leq \frac{2^{n}}{{2n\choose n}} \bar{r}^{n}.
\end{equation*}
Using Stirling's formula, we find that there exist $n_{0}\in\N$ and $\bar \varrho<1$ such that for all $n\geq n_{0}$ we have that
\begin{equation*}
\mu^{n}(x)\leq \bar \varrho^{n}.
\end{equation*}
Moreover, we have
\begin{equation*}
\sup_{1\leq n \leq n_{0}} \sup_{x} \mu^{n}(x)<1.
\end{equation*}
Indeed, otherwise, $\mu$ would be a Dirac mass and therefore its support would not generate a non-amenable group.
\end{proof}

\subsection{Surface groups}
In general, the geometry of cone types is hardly understood. In order to avoid artificial conditions, we will focus on hyperbolic surface groups. A surface group is the fundamental group of a  closed and orientable surface of genus $2$ or more. The standard presentation for an (orientable)  surface group of genus $g$ is
\[
\langle a_{1}^{\pm 1},b_{1}^{\pm 1},\ldots, a_{g}^{\pm 1}, b_{g}^{\pm 1} \mid \prod_{i=1}^{g} a_{i}b_{i}a_{i}^{-1}b_{i}^{-1}\rangle.
\]
Its Cayley  $2$-complex is the $2$-complex such that the one-skeleton is given by the Cayley graph $X$,  and the $2$-cells are bounded by loops in $X$ labeled by the relations. A surface group with standard presentation is planar, \emph{i.e.,} its 2-complex is homeomorphic to the hyperbolic disc. A strongly automatic structure of a surface group can be given explicitly, \emph{e.g.,} see \cite[Section 5.2]{GL:11}, and in particular there exists an automatic structure associated to the standard generating set that satisfies Assumption \ref{ass:1}. The planarity of the Cayley $2$-complex allows moreover a neat description of the cones and their boundaries.

\begin{lem}\label{lem:coneshape}
Let $(\Gamma,S)$ be surface group with standard generating set. Then, there exists an automatic structure $\mathcal{A}$ that satisfies Assumption \ref{ass:1}. Moreover, any cone type of $\A$ is large and is bounded by two geodesic rays starting from
the neutral element. 
\end{lem}

\begin{proof}
We refer to \cite[Section 5.2]{GL:11} for the fact that there exists an automatic structure that satisfies Assumption \ref{ass:1}. 
Let $x\in \Gamma\setminus\{e\}$ and  $C_\mathcal{A}(x)$ its cone defined by the automaton $\mathcal{A}$.  Since the 2-complex is homeomorphic to the plane, it can be endowed with an orientation.   Let $r_1,r_2:\R_+\to X$ be two infinite rays going through $x$ and which coincide up to $x$; let
$c_1,c_2$ be the geodesic rays extracted from $r_1,r_2$ starting at $x$. 
Let $V$ be a component of $X\setminus(c_1\cup c_2)$ which does not
contain $e$. Let us prove that $V$ is contained in $C_\mathcal{A}(x)$: let $y\in V$, and let us consider a segment $c_y$ joining $e$ to $y$.
Since the 2-complex is planar, Jordan's theorem implies that  $c_y$ has to intersect $\partial V$ at a point $z$, 
hence $c_1$ or $c_2$  beyond $x$.
Let us assume that it intersects $c_1$. Since $c_1$ is geodesic, we may replace the portion of $c_y$ before $z$ by $c_1$: 
it follows that the concatenation of $c_1$ up to $z$ and $c_y$ from $z$ to $y$ is geodesic; this implies that $y\in C_\mathcal{A}(x)$.

By Arzela-Ascoli's theorem and the planarity of the graph, we may find two rays $c_{\ell}$ and $c_{r}$ going through $x$ such that 
$C_\mathcal{A}(x)$ is the union of  those rays with all the components of their complement which do not contain $e$. 
\end{proof}

\section{Renewal structure and applications}\label{sec:ren}

\subsection{The construction}
In this section we assume that $\Gamma$ is a non-elementary hyperbolic group endowed with a finite generating set $S$ such that there is
a ubiquitous large cone type $\bT$.
The aim of the following part is  to construct a sequence of renewal times $R_{n}$ on which the random walk visits the root of a cone of type $\bT$ that it will never leave again. 

The main idea behind the construction is quite natural and is first sketched  informally; we also refer to  Figures \ref{fig:first} and \ref{fig:nth} for an illustration. 
The trajectory of the walk will be decomposed into parts of two different types: the ``{\bf e}xploring'' and the ``{\bf d}eciding'' parts. Though, let us fix a large ubiquitous cone type $\bT$ and start a random walk in the origin $e$. After some random time $E$ the random walk will visit a vertex of type $\bT$. At this point the walk may stay  in this cone forever or may leave it after some finite random time $D$. In the first case, we set $E$ to be the first renewal time. In the second case, the random walk after having left the cone at time $D$ will explore the underlying group in order to find another vertex of type $\bf T$. This procedure continues  until the walk decides to stay eventually in one cone of type $\bT$ and hence the first renewal point and renewal time are fixed. The construction of the subsequent renewal points is analogous. Eventually, this procedure decomposes the trajectory into aligned  and independently distributed pieces. However, the distribution of the first piece differs from the distributions of the subsequent ones, since the law of these latter pieces is given by the law conditioned to stay in the cones of the previous renewal points.

%

%\scalebox{1} % Change this value to rescale the drawing.

\begin{figure}[h]
\begin{minipage}[t]{5.5cm}
\begin{pspicture}(0,-2.487)(7.18,2.341)
\pscircle[linewidth=0.0139999995,dimen=outer](2.56,-0.073){2.414}
\psdots[dotsize=0.12](2.56,-0.073)
\psdots[dotsize=0.12](2.96,0.327)
\psarc[linewidth=0.024](1.55,0.937){1.55}{-24.928474}{58.799484}
\rput{-90.0}(4.633,2.827){\psarc[linewidth=0.024](3.73,-0.903){1.43}{124.31509}{211.93068}}
\pscustom[linewidth=0.0139999995]
{
\newpath
\moveto(2.98,0.307)
\lineto(3.0,0.347)
\curveto(3.01,0.367)(3.035,0.402)(3.05,0.417)
\curveto(3.065,0.432)(3.09,0.467)(3.1,0.487)
\curveto(3.11,0.507)(3.13,0.542)(3.14,0.557)
\curveto(3.15,0.572)(3.16,0.612)(3.16,0.637)
\curveto(3.16,0.662)(3.17,0.712)(3.18,0.737)
\curveto(3.19,0.762)(3.225,0.797)(3.25,0.807)
\curveto(3.275,0.817)(3.325,0.832)(3.35,0.837)
\curveto(3.375,0.842)(3.425,0.847)(3.45,0.847)
\curveto(3.475,0.847)(3.525,0.842)(3.55,0.837)
\curveto(3.575,0.832)(3.61,0.842)(3.62,0.857)
\curveto(3.63,0.872)(3.635,0.937)(3.63,0.987)
\curveto(3.625,1.037)(3.605,1.112)(3.59,1.137)
\curveto(3.575,1.162)(3.56,1.212)(3.56,1.237)
\curveto(3.56,1.262)(3.58,1.297)(3.6,1.307)
\curveto(3.62,1.317)(3.67,1.332)(3.7,1.337)
\curveto(3.73,1.342)(3.785,1.352)(3.81,1.357)
\curveto(3.835,1.362)(3.845,1.372)(3.83,1.377)
\curveto(3.815,1.382)(3.775,1.387)(3.75,1.387)
\curveto(3.725,1.387)(3.68,1.397)(3.66,1.407)
\curveto(3.64,1.417)(3.605,1.447)(3.59,1.467)
\curveto(3.575,1.487)(3.565,1.552)(3.57,1.597)
\curveto(3.575,1.642)(3.6,1.697)(3.62,1.707)
\curveto(3.64,1.717)(3.675,1.737)(3.69,1.747)
\curveto(3.705,1.757)(3.74,1.782)(3.76,1.797)
}
\usefont{T1}{ptm}{m}{n}
\rput(1.84,-0.768){$e$}
\psline[linewidth=0.0139999995cm,arrowsize=0.05291667cm 2.0,arrowlength=1.4,arrowinset=0.4]{->}(1.9,-0.533)(2.46,-0.133)
\psline[linewidth=0.0139999995cm,arrowsize=0.05291667cm 2.0,arrowlength=1.4,arrowinset=0.4]{->}(3.7,-0.233)(3.02,0.247)
\usefont{T1}{ptm}{m}{n}
\rput(3.82,-0.448){$Z_{R_1}$}
\pscustom[linewidth=0.0139999995]
{
\newpath
\moveto(2.58,-0.073)
\lineto(2.6,-0.023)
\curveto(2.61,0.0019999999)(2.64,0.037)(2.66,0.047)
\curveto(2.68,0.057)(2.725,0.072)(2.75,0.077)
\curveto(2.775,0.082)(2.82,0.082)(2.84,0.077)
\curveto(2.86,0.072)(2.885,0.042)(2.89,0.016999999)
\curveto(2.895,-0.0080)(2.9,-0.058)(2.9,-0.083)
\curveto(2.9,-0.108)(2.885,-0.148)(2.87,-0.163)
\curveto(2.855,-0.178)(2.825,-0.213)(2.81,-0.233)
\curveto(2.795,-0.253)(2.765,-0.283)(2.75,-0.293)
\curveto(2.735,-0.303)(2.69,-0.323)(2.66,-0.333)
\curveto(2.63,-0.343)(2.57,-0.358)(2.54,-0.363)
\curveto(2.51,-0.368)(2.455,-0.373)(2.43,-0.373)
\curveto(2.405,-0.373)(2.365,-0.353)(2.35,-0.333)
\curveto(2.335,-0.313)(2.31,-0.273)(2.3,-0.253)
\curveto(2.29,-0.233)(2.28,-0.188)(2.28,-0.163)
\curveto(2.28,-0.138)(2.28,-0.088)(2.28,-0.063)
\curveto(2.28,-0.038)(2.28,0.012)(2.28,0.037)
\curveto(2.28,0.062)(2.26,0.092)(2.24,0.097)
\curveto(2.22,0.102)(2.185,0.122)(2.17,0.137)
\curveto(2.155,0.152)(2.125,0.177)(2.11,0.187)
\curveto(2.095,0.197)(2.065,0.217)(2.05,0.227)
\curveto(2.035,0.237)(2.015,0.227)(2.01,0.207)
\curveto(2.005,0.187)(2.0,0.132)(2.0,0.097)
\curveto(2.0,0.062)(2.02,0.012)(2.04,-0.0030)
\curveto(2.06,-0.018)(2.1,-0.043)(2.12,-0.053)
\curveto(2.14,-0.063)(2.185,-0.073)(2.21,-0.073)
\curveto(2.235,-0.073)(2.28,-0.063)(2.3,-0.053)
\curveto(2.32,-0.043)(2.345,-0.013)(2.35,0.0069999998)
\curveto(2.355,0.027)(2.36,0.072)(2.36,0.097)
\curveto(2.36,0.122)(2.36,0.172)(2.36,0.197)
\curveto(2.36,0.222)(2.36,0.272)(2.36,0.297)
\curveto(2.36,0.322)(2.38,0.367)(2.4,0.387)
\curveto(2.42,0.407)(2.46,0.432)(2.48,0.437)
\curveto(2.5,0.442)(2.535,0.432)(2.55,0.417)
\curveto(2.565,0.402)(2.575,0.362)(2.57,0.337)
\curveto(2.565,0.312)(2.55,0.267)(2.54,0.247)
\curveto(2.53,0.227)(2.52,0.182)(2.52,0.157)
\curveto(2.52,0.132)(2.545,0.107)(2.57,0.107)
\curveto(2.595,0.107)(2.645,0.117)(2.67,0.127)
\curveto(2.695,0.137)(2.72,0.172)(2.72,0.197)
\curveto(2.72,0.222)(2.725,0.272)(2.73,0.297)
\curveto(2.735,0.322)(2.74,0.372)(2.74,0.397)
\curveto(2.74,0.422)(2.75,0.467)(2.76,0.487)
\curveto(2.77,0.507)(2.79,0.547)(2.8,0.567)
\curveto(2.81,0.587)(2.84,0.617)(2.86,0.627)
\curveto(2.88,0.637)(2.915,0.627)(2.93,0.607)
\curveto(2.945,0.587)(2.96,0.542)(2.96,0.517)
\curveto(2.96,0.492)(2.96,0.442)(2.96,0.417)
\curveto(2.96,0.392)(2.96,0.362)(2.96,0.347)
}
\usefont{T1}{ptm}{m}{n}
\rput(5.73,1.852){$C_{\mathcal{A}}(Z_{R_{1}})$}
\psline[linewidth=0.0139999995cm,arrowsize=0.05291667cm 2.0,arrowlength=1.4,arrowinset=0.4]{->}(4.62,1.627)(4.12,1.247)
\end{pspicture} \caption{First renewal step}\label{fig:first}
\end{minipage}
\hfill
\begin{minipage}[t]{10cm}
\begin{pspicture}(0,-1.764145)(9.502,2.8464108)
\psarc[linewidth=0.024](3.96,-2.69){3.0}{18.217094}{75.25644}
\psarc[linewidth=0.024](3.0,2.67){3.0}{-54.344673}{2.7591076}
\psarc[linewidth=0.024](6.62,-2.67){3.0}{18.217094}{75.25644}
\psarc[linewidth=0.024](5.66,2.69){3.0}{-54.344673}{2.7591076}
\psdots[dotsize=0.12](4.76,0.23)
\psdots[dotsize=0.12](7.42,0.25)
\pscustom[linewidth=0.024,linestyle=dashed,dash=0.16cm 0.16cm]
{
\newpath
\moveto(3.0,0.73)
\lineto(3.04,0.74)
\curveto(3.06,0.745)(3.105,0.765)(3.13,0.78)
\curveto(3.155,0.795)(3.185,0.84)(3.19,0.87)
\curveto(3.195,0.9)(3.205,0.96)(3.21,0.99)
\curveto(3.215,1.02)(3.23,1.075)(3.24,1.1)
\curveto(3.25,1.125)(3.27,1.18)(3.28,1.21)
\curveto(3.29,1.24)(3.33,1.295)(3.36,1.32)
\curveto(3.39,1.345)(3.445,1.375)(3.47,1.38)
\curveto(3.495,1.385)(3.55,1.39)(3.58,1.39)
\curveto(3.61,1.39)(3.65,1.375)(3.66,1.36)
\curveto(3.67,1.345)(3.685,1.305)(3.69,1.28)
\curveto(3.695,1.255)(3.705,1.21)(3.71,1.19)
\curveto(3.715,1.17)(3.75,1.135)(3.78,1.12)
\curveto(3.81,1.105)(3.87,1.075)(3.9,1.06)
\curveto(3.93,1.045)(3.995,1.03)(4.03,1.03)
\curveto(4.065,1.03)(4.125,1.03)(4.15,1.03)
\curveto(4.175,1.03)(4.22,1.04)(4.24,1.05)
\curveto(4.26,1.06)(4.28,1.1)(4.28,1.13)
\curveto(4.28,1.16)(4.295,1.21)(4.31,1.23)
\curveto(4.325,1.25)(4.355,1.295)(4.37,1.32)
\curveto(4.385,1.345)(4.435,1.38)(4.47,1.39)
\curveto(4.505,1.4)(4.565,1.43)(4.59,1.45)
\curveto(4.615,1.47)(4.67,1.495)(4.7,1.5)
\curveto(4.73,1.505)(4.785,1.51)(4.81,1.51)
\curveto(4.835,1.51)(4.87,1.485)(4.88,1.46)
\curveto(4.89,1.435)(4.9,1.385)(4.9,1.36)
\curveto(4.9,1.335)(4.905,1.285)(4.91,1.26)
\curveto(4.915,1.235)(4.94,1.19)(4.96,1.17)
\curveto(4.98,1.15)(5.025,1.13)(5.05,1.13)
\curveto(5.075,1.13)(5.135,1.13)(5.17,1.13)
\curveto(5.205,1.13)(5.27,1.13)(5.3,1.13)
\curveto(5.33,1.13)(5.39,1.14)(5.42,1.15)
\curveto(5.45,1.16)(5.515,1.175)(5.55,1.18)
\curveto(5.585,1.185)(5.655,1.19)(5.69,1.19)
\curveto(5.725,1.19)(5.795,1.185)(5.83,1.18)
\curveto(5.865,1.175)(5.92,1.15)(5.94,1.13)
\curveto(5.96,1.11)(5.985,1.06)(5.99,1.03)
\curveto(5.995,1.0)(6.0,0.945)(6.0,0.92)
\curveto(6.0,0.895)(6.0,0.83)(6.0,0.79)
\curveto(6.0,0.75)(6.0,0.67)(6.0,0.63)
\curveto(6.0,0.59)(5.99,0.52)(5.98,0.49)
\curveto(5.97,0.46)(5.945,0.395)(5.93,0.36)
\curveto(5.915,0.325)(5.885,0.24)(5.87,0.19)
\curveto(5.855,0.14)(5.81,0.045)(5.78,0.0)
\curveto(5.75,-0.045)(5.695,-0.115)(5.67,-0.14)
\curveto(5.645,-0.165)(5.6,-0.215)(5.58,-0.24)
\curveto(5.56,-0.265)(5.505,-0.305)(5.47,-0.32)
\curveto(5.435,-0.335)(5.37,-0.37)(5.34,-0.39)
\curveto(5.31,-0.41)(5.245,-0.43)(5.21,-0.43)
\curveto(5.175,-0.43)(5.11,-0.43)(5.08,-0.43)
\curveto(5.05,-0.43)(4.99,-0.43)(4.96,-0.43)
\curveto(4.93,-0.43)(4.865,-0.43)(4.83,-0.43)
\curveto(4.795,-0.43)(4.715,-0.43)(4.67,-0.43)
\curveto(4.625,-0.43)(4.545,-0.43)(4.51,-0.43)
\curveto(4.475,-0.43)(4.415,-0.42)(4.39,-0.41)
\curveto(4.365,-0.4)(4.32,-0.37)(4.3,-0.35)
\curveto(4.28,-0.33)(4.24,-0.29)(4.22,-0.27)
\curveto(4.2,-0.25)(4.18,-0.205)(4.18,-0.18)
\curveto(4.18,-0.155)(4.18,-0.105)(4.18,-0.08)
\curveto(4.18,-0.055)(4.2,-0.0050)(4.22,0.02)
\curveto(4.24,0.045)(4.285,0.085)(4.31,0.1)
\curveto(4.335,0.115)(4.385,0.14)(4.41,0.15)
\curveto(4.435,0.16)(4.475,0.185)(4.49,0.2)
\curveto(4.505,0.215)(4.55,0.235)(4.58,0.24)
\curveto(4.61,0.245)(4.665,0.25)(4.69,0.25)
}
\pscustom[linewidth=0.024]
{
\newpath
\moveto(4.78,0.23)
\lineto(4.83,0.23)
\curveto(4.855,0.23)(4.895,0.24)(4.91,0.25)
\curveto(4.925,0.26)(4.965,0.28)(4.99,0.29)
\curveto(5.015,0.3)(5.06,0.32)(5.08,0.33)
\curveto(5.1,0.34)(5.145,0.345)(5.17,0.34)
\curveto(5.195,0.335)(5.235,0.31)(5.25,0.29)
\curveto(5.265,0.27)(5.3,0.245)(5.32,0.24)
\curveto(5.34,0.235)(5.395,0.23)(5.43,0.23)
\curveto(5.465,0.23)(5.505,0.26)(5.51,0.29)
\curveto(5.515,0.32)(5.525,0.385)(5.53,0.42)
\curveto(5.535,0.455)(5.54,0.515)(5.54,0.54)
\curveto(5.54,0.565)(5.56,0.6)(5.58,0.61)
\curveto(5.6,0.62)(5.64,0.61)(5.66,0.59)
\curveto(5.68,0.57)(5.72,0.53)(5.74,0.51)
\curveto(5.76,0.49)(5.78,0.445)(5.78,0.42)
\curveto(5.78,0.395)(5.77,0.345)(5.76,0.32)
\curveto(5.75,0.295)(5.73,0.25)(5.72,0.23)
\curveto(5.71,0.21)(5.685,0.175)(5.67,0.16)
\curveto(5.655,0.145)(5.635,0.105)(5.63,0.08)
\curveto(5.625,0.055)(5.63,0.01)(5.64,-0.01)
\curveto(5.65,-0.03)(5.685,-0.05)(5.71,-0.05)
\curveto(5.735,-0.05)(5.79,-0.05)(5.82,-0.05)
\curveto(5.85,-0.05)(5.905,-0.05)(5.93,-0.05)
\curveto(5.955,-0.05)(5.99,-0.035)(6.0,-0.02)
\curveto(6.01,-0.0050)(6.03,0.035)(6.04,0.06)
\curveto(6.05,0.085)(6.06,0.15)(6.06,0.19)
\curveto(6.06,0.23)(6.07,0.295)(6.08,0.32)
\curveto(6.09,0.345)(6.115,0.39)(6.13,0.41)
\curveto(6.145,0.43)(6.18,0.475)(6.2,0.5)
\curveto(6.22,0.525)(6.255,0.565)(6.27,0.58)
\curveto(6.285,0.595)(6.32,0.625)(6.34,0.64)
\curveto(6.36,0.655)(6.38,0.71)(6.38,0.75)
\curveto(6.38,0.79)(6.38,0.87)(6.38,0.91)
\curveto(6.38,0.95)(6.38,1.02)(6.38,1.05)
\curveto(6.38,1.08)(6.385,1.13)(6.39,1.15)
\curveto(6.395,1.17)(6.43,1.195)(6.46,1.2)
\curveto(6.49,1.205)(6.55,1.21)(6.58,1.21)
\curveto(6.61,1.21)(6.675,1.205)(6.71,1.2)
\curveto(6.745,1.195)(6.815,1.19)(6.85,1.19)
\curveto(6.885,1.19)(6.935,1.23)(6.95,1.27)
\curveto(6.965,1.31)(6.98,1.38)(6.98,1.41)
\curveto(6.98,1.44)(6.99,1.49)(7.0,1.51)
\curveto(7.01,1.53)(7.05,1.56)(7.08,1.57)
\curveto(7.11,1.58)(7.17,1.59)(7.2,1.59)
\curveto(7.23,1.59)(7.29,1.595)(7.32,1.6)
\curveto(7.35,1.605)(7.41,1.6)(7.44,1.59)
\curveto(7.47,1.58)(7.52,1.54)(7.54,1.51)
\curveto(7.56,1.48)(7.585,1.415)(7.59,1.38)
\curveto(7.595,1.345)(7.58,1.295)(7.56,1.28)
\curveto(7.54,1.265)(7.495,1.235)(7.47,1.22)
\curveto(7.445,1.205)(7.4,1.165)(7.38,1.14)
\curveto(7.36,1.115)(7.325,1.08)(7.31,1.07)
\curveto(7.295,1.06)(7.26,1.035)(7.24,1.02)
\curveto(7.22,1.005)(7.195,0.965)(7.19,0.94)
\curveto(7.185,0.915)(7.175,0.865)(7.17,0.84)
\curveto(7.165,0.815)(7.16,0.76)(7.16,0.73)
\curveto(7.16,0.7)(7.175,0.65)(7.19,0.63)
\curveto(7.205,0.61)(7.23,0.575)(7.24,0.56)
\curveto(7.25,0.545)(7.255,0.51)(7.25,0.49)
\curveto(7.245,0.47)(7.215,0.435)(7.19,0.42)
\curveto(7.165,0.405)(7.125,0.37)(7.11,0.35)
\curveto(7.095,0.33)(7.075,0.29)(7.07,0.27)
\curveto(7.065,0.25)(7.075,0.215)(7.09,0.2)
\curveto(7.105,0.185)(7.14,0.165)(7.16,0.16)
\curveto(7.18,0.155)(7.225,0.15)(7.25,0.15)
\curveto(7.275,0.15)(7.32,0.155)(7.34,0.16)
\curveto(7.36,0.165)(7.385,0.19)(7.39,0.21)
\curveto(7.395,0.23)(7.405,0.25)(7.42,0.25)
}
\pscustom[linewidth=0.024,linestyle=dashed,dash=0.16cm 0.16cm]
{
\newpath
\moveto(7.44,0.27)
\lineto(7.49,0.27)
\curveto(7.515,0.27)(7.565,0.27)(7.59,0.27)
\curveto(7.615,0.27)(7.655,0.285)(7.67,0.3)
\curveto(7.685,0.315)(7.725,0.34)(7.75,0.35)
\curveto(7.775,0.36)(7.815,0.38)(7.83,0.39)
\curveto(7.845,0.4)(7.88,0.405)(7.9,0.4)
\curveto(7.92,0.395)(7.96,0.38)(7.98,0.37)
\curveto(8.0,0.36)(8.03,0.325)(8.04,0.3)
\curveto(8.05,0.275)(8.085,0.235)(8.11,0.22)
\curveto(8.135,0.205)(8.185,0.19)(8.21,0.19)
\curveto(8.235,0.19)(8.285,0.19)(8.31,0.19)
\curveto(8.335,0.19)(8.385,0.195)(8.41,0.2)
\curveto(8.435,0.205)(8.475,0.235)(8.49,0.26)
\curveto(8.505,0.285)(8.525,0.34)(8.53,0.37)
\curveto(8.535,0.4)(8.54,0.47)(8.54,0.51)
\curveto(8.54,0.55)(8.55,0.625)(8.56,0.66)
\curveto(8.57,0.695)(8.58,0.76)(8.58,0.79)
\curveto(8.58,0.82)(8.61,0.855)(8.64,0.86)
\curveto(8.67,0.865)(8.725,0.875)(8.75,0.88)
\curveto(8.775,0.885)(8.815,0.875)(8.83,0.86)
\curveto(8.845,0.845)(8.875,0.81)(8.89,0.79)
\curveto(8.905,0.77)(8.93,0.73)(8.94,0.71)
\curveto(8.95,0.69)(8.97,0.65)(8.98,0.63)
\curveto(8.99,0.61)(9.025,0.575)(9.05,0.56)
\curveto(9.075,0.545)(9.12,0.525)(9.14,0.52)
\curveto(9.16,0.515)(9.19,0.54)(9.2,0.57)
\curveto(9.21,0.6)(9.23,0.675)(9.24,0.72)
\curveto(9.25,0.765)(9.27,0.835)(9.28,0.86)
\curveto(9.29,0.885)(9.3,0.95)(9.3,0.99)
\curveto(9.3,1.03)(9.3,1.1)(9.3,1.13)
\curveto(9.3,1.16)(9.315,1.2)(9.33,1.21)
\curveto(9.345,1.22)(9.38,1.235)(9.4,1.24)
\curveto(9.42,1.245)(9.465,1.25)(9.49,1.25)
}
\usefont{T1}{ptm}{m}{n}
\rput(4.91,-1.045){$Z_{R_{n-1}}$}
\psline[linewidth=0.0139999995cm,arrowsize=0.05291667cm 2.0,arrowlength=1.4,arrowinset=0.4]{->}(4.3,-0.81)(4.7,0.11)
\psline[linewidth=0.0139999995cm,arrowsize=0.05291667cm 2.0,arrowlength=1.4,arrowinset=0.4]{->}(7.5,-0.83)(7.42,0.07)
\usefont{T1}{ptm}{m}{n}
\rput(7.78,-1.065){$Z_{R_n}$}
\end{pspicture}  \caption{$n$th renewal step}\label{fig:nth}
\end{minipage}
\end{figure}

We now present the details of the construction.
Let $$E=\inf\{n\geq 0\,;\,T_{\mathcal{A}}(Z_{n})=\bT\}$$ the first time the random walk visits a vertex of type $\bT$. The random variable $E$ is a stopping time and a priori takes  values in $\N\cup\{\infty\}$. However, since $\bT$ is ubiquitous it can be shown that $E$ is almost surely finite, see proof of Lemma \ref{lem:R_1finite}. Recall that  $\theta$ is  the canonical shift on the space of trajectories $\mathcal{T}$ and  thus $$ E\circ \theta^{k}=\inf\{n\geq 0\,;\,T_{\mathcal{A}}(Z_{n+k})=\bT\}.$$
Define the stopping time
$$ D=\inf\{n\geq 1\,;\,Z_{n}\notin C_{\mathcal{A}}(Z_{0})\}$$ and  consider its shifted versions
$$ D\circ \theta^{k}=\inf\{n\geq 1\,;\,Z_{n+k}\notin C_{\mathcal{A}}(Z_{k})\}.$$
Observe that the random variables $D\circ \theta^{k}$ might be finite or infinite, see Lemma \ref{lem:stayinconegen}. We define 
\[R_{0}=0.\]
In order to define the subsequent renewal times we  introduce a sequence of stopping times $(S_{k}^{(1)})_{k\geq 0}$:
$$S_{0}^{(1)}=E~\mbox{and inductively}~S_{k+1}^{(1)}=S_{k}^{(1)}+ D_{k}^{(1)} + E_{k}^{(1)}\leq \infty,$$ where $$ D_{k}^{(1)} = D\circ \theta ^{{S^{(1)}_{k}}}\mbox{ and } E_{k}^{(1)}= E\circ\theta^{{S^{(1)}_{k}}+D_{k}^{(1)} }.$$
Letting 
\begin{equation*}\label{eq:K}
K^{(1)}=\inf\{k\geq 0\,;\,S_{k}^{(1)}<\infty, S_{k+1}^{(1)}=\infty\}\leq\infty
\end{equation*}
 we define the first renewal time
$$R_{1}=S_{K^{(1)}}^{(1)}\leq \infty.$$ 
Equivalently, this renewal time  can be written as 
$$ R_{1}=\inf\{{k\geq 0}\,;\,Z_{i}\in C_{\mathcal{A}}(Z_{k})~\forall i\geq k,~T_{\mathcal{A}}(Z_{k})=\bT \}.$$ 
In words, $R_{1}$ is the first time the random walk hits the root of a cone of type $\bT$  that it never leaves afterwards.
Note that $R_{1}$ is not a stopping time.

Inductively, we define the  $n$th renewal time. Provided that $R_{n-1}<\infty$ we define as above:
$$S^{(n)}_{0}=R_{n-1}+ 1+ E\circ \theta^{R_{n-1}+1}~\mbox{and inductively}~S_{k+1}^{(n)}=S_{k}^{(n)}+ D_{k}^{(n)} + E_{k}^{(n)}\leq \infty,$$ 
where 
$$ D_{k}^{(n)} = D\circ \theta ^{{S^{(n)}_{k}}}\mbox{ and } E_{k}^{(n)}= E\circ\theta^{{S^{(n)}_{k}}+D_{k}^{(n)} }.$$ 
Letting $K^{(n)}=\inf\{k\geq 0\,;\,S^{(n)}_{k}<\infty, S^{(n)}_{k+1}=\infty\}\leq\infty$ we can define the $n$th renewal time
$$R_{n}=S^{(n)}_{K^{(1)}}\leq \infty,$$ which is the same as 
$$ R_{n}=\inf\{k> R_{n-1}\,;\,Z_{i}\in C_{\mathcal{A}}(Z_{k})~\forall i\geq k,~T_{\mathcal{A}}(Z_{k})=\bT\}.$$ 

Without any further assumption, we have the following basic result.
\begin{thm}\label{thm:reg}
Let $(\Gamma, S)$ be a non-elementary hyperbolic group and $\mathcal{A}$ a corresponding automatic structure with a large ubiquitous 
cone type $\bT$. Let $\mu$ be a driving measure satisfying Assumption \ref{ass:2}. Then,  the renewal times $R_{n}$ are almost surely finite and 
$d(e,Z_{R_{n}})=\sum_{i=1}^{n} d(Z_{R_{i-1}}, Z_{R_{i}})$, where $d((Z_{R_{i-1}}, Z_{R_{i}}))_{i\geq 2}$ are i.i.d.~random variables.
\end{thm}

In order to prove Theorem \ref{thm:reg} we first prove two lemmata.
\begin{lem}\label{lem:R_1finite}
Under the assumption of Theorem \ref{thm:reg} the random variable  $R_1$ is almost surely finite under $\P_{x}$ for any $x\in\Gamma$.
\end{lem}

\begin{proof}
By  irreducibility of the random walk and ubiquity of $\bT$ 
there exist $c>0$ and  $m\in\N$ such that 
$$\P_{y}[\exists n\in [0, m-1]:~T(Z_{n})=\bT]>c>0 \mbox{ for all } y\in \Gamma.$$ 
Hence, by  the Markov property we  have
\begin{equation}\label{eq:Eexpmom}
\P_{y}[E\geq N m] \leq (1-c)^{N},
\end{equation}
and hence that $\P_{y}[E<\infty]=1$ for all $y\in\Gamma$. Fix $x\in\Gamma$.
Since we are dealing with stopping times, for any $y\in\Gamma$, the law of $E_k^{(1)}$
conditioned on $\{S_{k}^{(1)}<\infty,   D_{k}^{(1)}<\infty, Z_{S_{k}^{(1)} +  D_{k}^{(1)}}=y\}$ is  the law of $E$ under $\P_{y}$.
Therefore $$\P_{x}[E_{k}^{(1)}<\infty\mid S_{k}^{(1)}<\infty, D_{k}^{(1)}<\infty]=1.$$
Since the law of $D_{k}^{(1)}$ conditioned on $\{S_{k}^{(1)}<\infty, Z_{S_{k}^{(1)}}\}$ is the law of $D$ under $\P_{y}$ for $y$ such that $T(y)=\bT$ we obtain using Lemma \ref{lem:stayinconegen} that
 $$\P_{x}[ D_{k}^{(1)}<\infty\mid S_{k}^{(1)}<\infty]\leq 1-c.$$
 Therefore,
 \[
 \P_{x}[S_{k+1}^{(1)}<\infty \mid S_{k}^{(1)}<\infty ] =  \P_{x}[D_{k}^{(1)}<\infty \mid S_{k}^{(1)}<\infty ] \leq 1-c.
 \]
 Hence,  by the strong Markov property we obtain for all $N\in\N$ 
\begin{equation}\label{eq:lem:Rfinite2}
  \P_{x}[S^{(1)}_{k}<\infty~\forall k\leq N] \leq  (1-c)^{N}. 
\end{equation} and hence $\P_{x}[R_{1}=\infty]=\P_{x}[K^{(1)}=\infty]=0$. 
\end{proof}

A main feature of the definition of the cones is the following property: for  any $x,y\in\Gamma$ with same  cone type and any $A\subset \mathcal{T}$ we have that
\[
 \P_x[ (x^{-1}Z_{n})_{n\in \N }\in A \mid D=\infty]=\P_{y} [(y^{-1}Z_{n})_{n\in \N }\in A  \mid D=\infty].\]
 Therefore, we may introduce a new probability measure: for $A\subset \mathcal{T}$ let
 \[ \Q_{\bT}[ (Z_{n})_{n\in\N}\in A]= \P_{x}[ (x^{-1}Z_{n})_{n\in \N }\in A \mid D=\infty],
 \] where $x$ is of cone type $\bT$. We write $\E_{\bT}$ for the corresponding expectation.

 Define the  $\sigma$-algebras $$\GG_{n}=\sigma(R_{1},\ldots, R_{n}, Z_{0}, \ldots, Z_{R_{n}})\quad n\geq 1.$$
Although the $R_{n}$ are not stopping times we have the following ``Markov property''.
\begin{lem}\label{lem:markov}
Let $\bT$ be some ubiquitous large cone type. Then, for all $n\geq 1$ we have that $R_{n}$ is almost surely finite and for any measurable set $A\subset\T$ and any $y\in\Gamma$
$$ \P_{y}[ (Z_{R_{n}}^{-1}Z_{R_{n}+k})_{k\in\N}\in  A\mid \GG_{n}] = \Q_{\bT} [ (Z_{k})_{k\in\N}\in A].$$
\end{lem} 
 
\begin{proof} Without loss of generality let us assume that $y=e$. Besides the finiteness of the $R_{n}$ we have to check the definition of  the conditional expectation:  for all bounded $\GG_n$-measurable function $H$
and all measurable set $A\subset \T$ it holds that $$\E[H \1_{(Z_{R_n}^{-1}Z_{R_n+k})_{k}\in  A} ]= \Q_{\bT}[(Z_k)_k\in A]\cdot \E[H].$$

We will proceed by induction. So let us consider the case $n=1$. Lemma \ref{lem:R_1finite} implies that $\P[R_{1}<\infty]=1$. Now, we  observe that $\{R_1= S_l\}=\{S_l <\infty\}\cap \{D\circ{\theta^{S_l}}=\infty\}$.
Let $l\in\N$ and $x\in\Gamma$. Then, there exists (due to $\GG_{1}$-measurability) some random variable $H_{x,l}$ measurable with respect to $\{Z_{i}\}_{i\leq S_{l}}, S_{l}\}$ such that $H=H_{x,l}$ on the event $\{R_{1}=S_{l}, Z_{S_{l}}=x\}$.  Therefore, we may write
\begin{eqnarray*}
\E [H \1_{(Z_{R_1}^{-1}Z_{R_1+k})_k\in A}] 
&=& \sum_{l\geq 1} \sum_{x\in\Gamma} \E[\1_{S_{l}<\infty} \1_{D\circ \theta^{S_{l}}=\infty}   \1_{Z_{S_{l}}=x} \1_{{(Z_{S_l}^{-1}Z_{S_l+k})_k\in A}} H_{x,l} ] \cr
&=& \sum_{l\geq 1}\sum_{x\in\Gamma \atop T(x)=\bT} \E[\1_{S_{l}<\infty}  \1_{Z_{S_{l}}=x}  H_{x,l} ] \E_{x}[\1_{D=\infty} \1_{(x^{-1}Z_k)_k \in A}] \cr
&=& \Q_{\bT}[(Z_k)_k\in A]  \sum_{l\geq 1} \sum_{x\in\Gamma \atop T(x)=\bT} \E[\1_{S_{l}<\infty}  \1_{Z_{S_{l}}=x}  H_{x,l} ] \P_x[D=\infty]\,.
\end{eqnarray*} 
In the second equality we have used the strong Markov property since $S_l$ is a stopping time, and in the third we have applied the definition of $\Q_\bT$ as a conditional probability. 
Substituting in the above a trivial $A$ we have
$$ \E[H] = \sum_{l\geq 1} \sum_{x\in\Gamma \atop T(x)=\bT} \E[\1_{S_{l}<\infty}  \1_{Z_{S_{l}}=x}  H_{x,l} ] \P_x[D=\infty]$$
which shows that
  \begin{eqnarray*}\E [H \1_{(Z_{R_1}^{-1}Z_{R_1+k})_k\in A}] 
   & = &  \Q_{\bT}[(Z_k)_k\in A]\cdot  \E[H ].
   %\cr
  % & = & \E[\P_{T}[(Z_k)_k\in A| {D=\infty}] H \1_{R_{1}<\infty] \,.
  \end{eqnarray*} 
This concludes the proof for $n=1$ and implies the finiteness of $R_{2}$ since now $\P[R_{2}<\infty]=\Q_{\bT}[R_{1}<\infty]$ and due to Lemma \ref{lem:R_1finite} the latter probability is equal to one.

The induction proceeds similarly. If $H$ is $\GG_n$-measurable, then,  there exists some random variable $H_{x,l}$ measurable with respect to $\GG_{n-1}$ such that $H=H_{x,l}$ on the event $\{R_{n}=S^{(n)}_{l}, Z_{S^{(n)}_{l}}=x\}$, to which we may apply the induction hypotheses.
The computations are left to the reader. 
 \end{proof}

\begin{proof} (Theorem \ref{thm:reg})
By Lemma \ref{lem:markov} the  renewal times $R_{n}$ are almost surely finite and, by construction, all renewal points $Z_{R_{n}}$ lie on one geodesic.  Hence, $d(Z_{R_{n}},e)=\sum_{i=1}^{n} d(Z_{R_{i}}, Z_{R_{i-1}})$. Eventually, 
Lemma \ref{lem:markov} implies that $(d(Z_{R_{i}}, Z_{R_{i-1}}))_{i\geq 2}$ all have the same distribution and are independent.
\end{proof} 
 \begin{rem} 
The renewal structure  yields an alternative construction of the law of the walk. Let $Q_0$ be the law of $(Z_n;n\leq R_{1})$ under $\P_x$ and let $Q$ be the law of $(Z_{R_1}^{-1}Z_{(R_1+n)}; n\leq  R_2)$. We can obtain the measure $\P_x$ by choosing a path according to $Q_0$ and concatenate it with an i.i.d.~sequence sampled from $Q$.% The resulting measure is $\P_x$. %The proof of that requires Lemma \ref{lem:markov}; Theorem \ref{thm:reg} is not sufficient. 
 \end{rem}

%%%%%%%%%%%%%%%%%%%%%%%%%%%%%%%%%%%%%%%%%%%%%
%%%%%%%%%%%%%%%%%%%%%%%%%%%%%%%%%%%%%%%%%%%%%%%%%%%%%%%%%%%%%%%

 \section{Surface groups}
While the results in the previous sections are valid for random walks with finite first moments, we need some additional assumptions in order to prove a central limit theorem and  analyticity of the rate of escape and of the asymptotic variance. 

We say a real valued random variable $Y$ has exponential moments if $\E[\exp(\lambda Y)]<\infty$ for some $\lambda>0$, or equivalently, if  there exist positive constants $C$ and $c<1$ such that $\P[Y=n]\leq C c^{n} $ for all $n\in \N$.
The random variables appearing  in Theorem \ref{thm:reg} do in general not  have exponential moments. 
However, this is the case under the following assumption.

\begin{assumption}\label{ass:3}
The driving measure $\mu$ has exponential moments, \emph{i.e.}, $\E[\exp(\lambda_{\mu} d(X_{1},e))]<\infty$ 
for some $\lambda_{\mu}>0$.
\end{assumption}

In the sequel of this section we will only consider surface groups with standard generating sets. Lemma \ref{lem:coneshape} assures the existence of an automatic structure $\mathcal{A}$ with a  ubiquitous large cone type $\bT$. The latter allows the construction of the renewal points and times, see Section \ref{sec:ren}, and this construction depends on the choices  of $\mathcal{A}$ and $\bT$. However, in order to facilitate the reading, we formulate the statements without specifying the structure $\mathcal{A}$ nor the type $\bT$. 

\begin{lem}\label{lem:momentbounds} Let $(\Gamma,S)$ be a surface group with standard generating set.  Under Assumption \ref{ass:3} the renewal times $R_{1}$ and  $(R_{i+1}-R_{i})$ for $i\geq 1$ have exponential moments. The same holds true for $d(Z_{R_{1}},e)$ and $d(Z_{R_{i+1}},Z_{R_{i}})$ for $i \geq 1$.
\end{lem}

\begin{proof}
Let us first prove that $R_{1}$ has exponential moments. In Equation  (\ref{eq:Eexpmom})   we have established that $E$   has uniform exponential moments: there are some constants $\lambda_E>0$ and $C_E<\infty$  such that for all $x\in\Gamma$ we have $\E_x[\exp(\lambda_E E)]\le C_E$.

In order to control  the moments of $D$ we make use of the non-amenability and the planarity of $\Gamma$.  Due to Lemma \ref{lem:expdecay} there exists some $\varrho<1$ such that for all $x,y\in\Gamma$ and all $n\ge 1$, 
 \begin{equation}\label{eq:rho}
 \P_x[Z_n=y] \le  \varrho^n.
 \end{equation}
 
We proceed with the tails of $\P_{x}[D=n]$ for $x$ such that $T(x)=\bT$. Let $\delta>0$ to be chosen later, then
\begin{equation}\label{eq:lem:momentbound:1}
\P_{x}[D=n+1]\leq \P_{x}[d(Z_{n},x)\leq \delta n,  D=n+1]+ \P_{x}[d(Z_{n},x)\geq \delta n,  D=n+1].\end{equation}
The second summand is controlled by using the Chebyshev inequality:
\begin{eqnarray*}
\P_{x}[d(Z_{n},x) \geq \delta n,  D=n+1] & \leq & \P\left[\sum_{i=1}^{n} d(X_{i},e)\geq \delta n\right]\\
& \leq & \frac{\E[\exp(\lambda_{\mu} \sum_{i=1}^{n} d(X_{i},e))] }{\exp(\lambda_{\mu}\delta n)}\\
& = & \frac{(\E[\exp(\lambda_{\mu}  d(X_{1},e))] )^{n}}{\exp(\lambda_{\mu}\delta n)}.\end{eqnarray*}
Since $\mu$ has exponential moments we can choose $\delta$ sufficiently large such that the latter term converges exponentially fast to $0$.

In order to treat the first summand of Equation (\ref{eq:lem:momentbound:1}) we make use of Lemma \ref{lem:coneshape}.
Let $\gamma$ be a geodesic. We define the $m$-tube of $\gamma$ as $\gamma^{(m)}:=\bigcup_{x\in \gamma} B(x,m)$. 
Let $\gamma_{\ell}, \gamma_{r}$ be the two geodesics such that $\partial_{\Gamma} C(x) = \gamma_{\ell}\cup \gamma_{r}$, 
then we define the $m$-tube of $\partial_{\Gamma} C(x)$ as $\partial^{(m)} C:= \gamma_{\ell}^{(m)}\cup \gamma_{r}^{(m)}$. 
 Now we obtain, using mainly Equation (\ref{eq:rho}) and the linear growth of the boundary of cones, that for all $\eps>0$:
\begin{eqnarray*}
\P_{x}\left[d(Z_{n},x) \leq \delta n, D=n+1\right] 
  &\leq &  \P_{x}[ d(Z_{n},x) \leq \delta n,  Z_{n}\in \partial^{(\eps n)} C(x)] \cr 
     & &+  \P_{x}[ D=n+1, Z_{n}\notin \partial^{(\eps n)} C(x)] \cr
  &\leq &  \varrho^{n} |\partial^{(\eps n)} C(x)\cap B(x,\delta n)| + \P[d(Z_{n+1},Z_{n})>\eps n]\cr
  &\leq&  \varrho^{n} (2\delta n) |S|^{\eps n}  + \P[d(X_{1},e)>\eps n].
 \end{eqnarray*}
Choose eventually $\eps>0$ sufficiently small so that $\varrho |S|^{\eps}<1$. Since for $\eps$ fixed the probability $\P[d(X_{1},e)>\eps n]$ decays exponentially in $n$,
there are some constants $\lambda_D>0$ and $C_D<\infty$ such that 
\begin{equation}\label{eq:PzD}
\E_{x}[\exp(\lambda_D D) \1_{\{D<\infty\}}] \le C_D\quad \forall x: T(x)=\bT.
\end{equation}

Now, recall that 
\begin{equation}\label{eq:R1}
R_{1}= E+\sum_{k=1}^{K^{(1)}} (D_{k}^{(1)}+E_{k}^{(1)}),\end{equation}
where $K^{(1)}$ is the smallest time $k$ such that $D^{(i)}_{k+1}=\infty$.

Recall that $K^{(1)}$ has exponential moments, see (\ref{eq:lem:Rfinite2}), so that there exist constants
$\lambda_K>0$ and $C_K<\infty$ such that $\P[K^{(1)}=k] \le C_K \exp(-\lambda_{K}k)$.
We can decompose
\begin{equation*}
\P_x[n\leq D+E\circ\theta^D<\infty]  \leq  \P_{x}[n/2\leq D <\infty] + \sum_{k=1}^{n/2}\P_x[ E\circ\theta^k \geq n/2] .\end{equation*}
Hence, we may find constants $C>0$, $\lambda>0$ such that
$$\E_x[\exp(\lambda(D+ E\circ\theta^D))\1_{\{D <\infty\}}] \le C \quad \forall x: T(x)=\bT.$$
Therefore, we may choose $\lambda_1$ small enough such that
\[ \E_{x}[\exp(\lambda_1(D+ E\circ\theta^D)) | D<\infty] \le \exp(\lambda_{K} /2)\quad \forall x: T(x)=\bT.\]
Eventually, using the strong Markov property, 
\begin{eqnarray*}
\E[\exp(\lambda_{1} R_{1})] & = &\sum_{k=1}^{\infty} 
  \E\left[\exp\left(\lambda_{1}\left(E+\sum_{i=1}^{k} (D_{i}^{(1)}+E_{i}^{(1)})\right)\right)\mid K^{(1)}=k\right] \P[K^{(1)}=k]\cr
  & \leq & C_K\sum_{k=1}^{\infty} \E[\exp{\lambda_{1} E}] (\exp(\lambda_{K}/2))^{k}\exp(-\lambda_K k)\cr
  &\leq& C \sum_{k=1}^{\infty} \exp(-\lambda_{K}k/2) <\infty.
\end{eqnarray*}

The proof for $R_{i+1}-R_{i},$ $i\geq 1$, is analogous since  the laws of the different
$D_k^{(i+1)}$ are independent of $R_i$.
Moreover, Equation (\ref{eq:Eexpmom}) implies exponential
moments for $E\circ \theta^{R_{i}}$ and $E_k^{(i+1)}$ as well. 
\bigskip

We turn to the exponential moments of the distances between two successive renewal points. Let $\delta>0$ to be chosen later. 
Then, since $R_{1}$ has exponential moments,
\begin{eqnarray*}
\P[ d(Z_{R_{1}},e)\geq k] &\leq & \P[d(Z_{R_{1}},e)\geq k, R_{1}\geq k\delta] + \P[d(Z_{R_{1}},e)\geq k, R_{1}\leq k\delta]\cr
%&\leq & C e^{-k c \delta} +  \P[d(Z_{R_{1}},e)\geq k, R_{1}\leq k\delta]\cr
&\leq & C e^{-c k  \delta} + \P\left[\sum_{i=1}^{k\delta} d(X_{i},e)\geq k\right]
\end{eqnarray*}
and hence,  using again Chebyshev's inequality,  we see that for suitable $\delta$ 
the last term decays exponentially fast. The proof for $d(Z_{R_{i}}, Z_{R_{i+1}})$, $i\geq 1$, is in the same spirit: 
\begin{eqnarray}
\P[ d(Z_{R_{i+1}}, Z_{R_{i}})\geq k] &\leq & \P[ R_{i+1}-R_{i}\geq k\delta]
 + \P[d(Z_{R_{i+1}}, Z_{R_{i}})\geq k, R_{i+1}-R_{i}\leq k\delta]\cr
&\leq & C e^{-ck\delta} +  \P[d(Z_{R_{i+1}}, Z_{R_{i}})\geq k, R_{i+1}-R_{i}\leq k\delta].
\end{eqnarray}
Using Lemma \ref{lem:markov} the last summand becomes
$$ \P[d(Z_{R_{i+1}}, Z_{R_{i}})\geq k, R_{i+1}-R_{i}\leq k\delta] = \Q_{\bT}[d(Z_{R_1},Z_{0})\geq k, R_{1}\leq k\delta]$$
Once again, an application of the exponential Chebyshev inequality yields that 
$\delta$ can be chosen such that 
$\Q_{\bT}\left[\sum_{i=1}^{k\delta} d(X_{i},e)\geq  k\right]$ decays exponentially fast and the claim follows as above. 
\end{proof}
\begin{rem}\label{rem:inftyends}
In the case of hyperbolic groups with infinitely many ends, it follows from Stalling's splitting theorem that the boundaries
of the cones are finite if we choose the generators accordingly. 
Hence the proof of Lemma \ref{lem:momentbounds} applies to this setting.
\end{rem}

\begin{cor}\label{cor:overshoot}
Set
$$M_{k}=\sup\{d(Z_{n},Z_{R_{k}}), R_{k}\le n \leq R_{k+1}\}, ~k\geq 1,$$
and $$k(n)=\sup\{k:~R_{k}\le n\}\,.$$ 
Under the assumptions of Lemma \ref{lem:momentbounds}, 
$(M_{k})_{k\geq 1}$ is an i.i.d.~sequence with exponential moments and
$$ \frac{n}{k(n)} \xrightarrow[n \to \infty]{a.s.}  \E[R_{2}-R_{1}]<\infty.$$
\end{cor}

\begin{proof} 
The proof that $(M_k)$ are i.i.d.~follows from Lemma \ref{lem:markov}, as in the proof of  Theorem  \ref{thm:reg}. The fact that $M_k$ have exponential moments can either be seen as in Lemma \ref{lem:momentbounds} or as follows. Let $\delta>0$ (to be chosen later). Then, by the law of total probability, for $m\in\N$
\begin{eqnarray*}
 \P[M_{k}\geq m] &\leq& \P[R_{k+1}-R_{k}\geq m\delta ] + \P[\sup\{ d(Z_{R_{k}},Z_{n}), R_{k}\leq  n\leq  R_{k}+m \delta\}\geq m]\cr
 & \leq & \P[R_{k+1}-R_{k}\geq m \delta] + \Q_{T}\left[ \sum_{i=1}^{m \delta} d(e,X_{i})\geq m\right].
\end{eqnarray*}
Since $R_{k+1}-R_{k}$ and $\mu$ have exponential moments, applications of Chebyshev's inequality for both summands show that  we can choose $\delta$  sufficiently small such that $\P[M_{k}\geq m]$ decays exponentially fast to $0$.
Concerning $k(n)$, we write 
$$\frac{n}{k(n)}= \frac{n}{R_{k(n)}}\frac{R_{k(n)}}{k(n)}\,.$$
The second factor tends a.s.~to $\E[R_2-R_1]$ by the strong law of large numbers (since $k(n)$ tends to infinity).
For the first factor we observe  that $R_{k(n)}\le n \le R_{k(n)+1}$,
hence $$\limsup_{n\to\infty} \frac{R_{k(n)}}{n}\le 1\,.$$
On the other hand, since $n\ge k(n)$ and  $(R_{k(n)}-R_{k(n)+1})$ have finite moments,
$$\lim_{n\to\infty} \frac{R_{k(n)}-R_{k(n)+1}}{n}=0\quad \hbox{a.s.}$$ and hence
$$\liminf_{n\to\infty}\frac{R_{k(n)}}{n}\ge \liminf_{n\to\infty}\left( \frac{R_{k(n)}-R_{k(n)+1}}{n}\right)+ \frac{R_{k(n)+1}}{n}\ge 1\,.$$
\end{proof}

\subsection{Limit Theorems}
The existence of the law of large numbers (LLN) is a direct consequence of  Kingman's subadditive ergodic theorem. Moreover, it was proven by Guivarc'h \cite{Gui:80} that for non-amenable Cayley graphs the speed is positive.
We give a formula for the speed in terms of the renewal structure and recover the above results without using Kingman's theorem for driving measures with exponential moments.

\begin{thm}\label{thm:LLN_planar}
Let $(\Gamma,S)$ be a surface group with standard  generating set  and  assume the driving measure $\mu$ to have exponential moments. Then,
\begin{equation}\label{eq:v} 
\frac1n d(Z_{n},e)  \xrightarrow[n \to \infty]{a.s.} v=  \frac{\E[ d(Z_{ R_{2}}, Z_{ R_{1}})]}{\E[  R_{2}- R_{1}]}>0.
\end{equation}
\end{thm}
\begin{proof}

% {\tiny The proof is  standard. 
The law of large numbers for i.i.d.~sequences applied to $(R_{k+1}-R_k)_k$ and  $(d(Z_{R_k}, Z_{R_{k+1}}))_k$, tells us that
\begin{equation*} \frac{R_{k}}k  \xrightarrow[k \to \infty]{a.s.} \E[R_{2}-R_{1}] \mbox{ and } 
\frac{d(Z_{R_{k}},e)}{k}  \xrightarrow[k \to \infty]{a.s.} \E [d(Z_{R_{2}},Z_{R_{1}})].
\end{equation*}
With $k(n)=\max\{k:~R_{k}\leq n\}$ we have $k(n)/n \to 1/\E[R_{2}-R_{1}]$ a.s.~by Corollary \ref{cor:overshoot}. 
Moreover, the latter also implies that
$$\lim_{n\to\infty}\frac{d(Z_{n},e)-d(Z_{R_{k(n)}},e)}n  \le \lim_{k\to\infty}\frac{M_k}k=0 \quad\hbox{a.s.}\,.$$
Hence
\begin{eqnarray*}
\frac{d(Z_{n},e)}n& =& \frac{d(Z_{n},e)-d(Z_{R_{k(n)}},e)}n + \frac{d(Z_{R_{k(n)}},e)}{k(n)} \frac{k(n)}{n}\cr
&   \xrightarrow[n \to \infty]{a.s.} & 0  + \frac{\E[ d(Z_{ R_{2}}, Z_{ R_{1}})]}{\E[ R_{2}- R_{1}]}.
\end{eqnarray*} The strict positiveness of $v$ follows from the fact that $\E[ R_{2}- R_{1}]<\infty.$
\end{proof}

\subsection{Proof of Theorem \ref{thm:CLT_planar}}
Consider the following sequence of real valued random variables: 
$$\xi_{i}=d(Z_{R_{i+1}}, Z_{R_{i}}) - (R_{i+1}- R_{i})v\quad i\geq 1.$$ 
According to Theorem \ref{thm:reg}  this is a sequence of centered  i.i.d.~random variables. 
Moreover, $\Sigma=\E[\xi_{1}^2]>0$ since $\P( |\xi_{i}|>k)>0$ for some $k\ge 0$.
 Let
$$ S_{n}=\sum_{i=1}^n \xi_{i}\mbox{, and } \Sigma=\E[\xi_{1}^2],$$  
The sequence $(S_{n})_n$ does not only satisfy a central limit theorem, \emph{i.e.}, 
$S_{n}/\sqrt{n}\stackrel{\D}{\longrightarrow} \normal(0,  \Sigma)$, but also an invariance principle, \emph{i.e.}, 
$\frac1{\Sigma\sqrt{n}} S_{\lfloor nt\rfloor}$ converges in distribution to a standard Brownian motion (\emph{e.g.,} see Donsker's Theorem 14.1
 in \cite{billingsley}).
Let $$k(n)=\max\{k:~\sum_{i=1}^k (R_{i}-R_{i-1})<n\}$$ as in Corollary \ref{cor:overshoot}.

As the invariance principle is preserved under change of time (\emph{e.g.,} see Theorem 14.4 in \cite{billingsley}) the sequence
$\frac1{\Sigma\sqrt{k(n)}} S_{\lfloor k(n)t\rfloor}$ also converges in distribution to a standard Brownian motion. Choosing $t=1$ yields in particular
$$S_{k(n)}/\Sigma\sqrt{k(n)} \xrightarrow[n \to \infty]{\D}\normal(0,  1)\,.$$
From Corollary \ref{cor:overshoot}, since $(n/k(n))$ tends to a constant almost surely, we get
that $$S_{k(n)}/\Sigma\sqrt{n}\xrightarrow[n \to \infty]{\D}
\normal(0,  \sigma^2),$$ where $$\sigma^2= 1 / \E[R_{2}-R_{1}].$$ 
Corollary \ref{cor:overshoot} also ensures that the random variables
$$M_{k}=\sup\{d(Z_{n},Z_{R_{k}}), R_{k}\le n \leq R_{k+1}\}, ~k\geq 1,$$
form an i.i.d.~sequence with exponential moments.
Now, for any positive $\eta$
\begin{eqnarray}\label{eq:thm:clt:final1} 
\P\left[| S_{k(n)}- (d(Z_{n},e)-nv)| > \eta \sqrt{n}\right] 
   & \leq & \P\left[ d(Z_{R_{k(n)+1}},e)-d(Z_{n},e) \geq \frac\eta{2} \sqrt{n}\right] \cr 
      & & + \P\left [ v(n-R_{k(n)+1}) \geq \frac\eta{2} \sqrt{n}\right].
\end{eqnarray}
We start by  treating the first summand in (\ref{eq:thm:clt:final1}). Let $M_{0}=\sup\{d(Z_{R_{1}},Z_{0}),  n \leq R_{1}\}$, then
\begin{eqnarray*}
\P\left[ d(Z_{R_{k(n)+1}},e)-d(Z_{n},e) \geq \frac\eta{2} \sqrt{n}\right] 
 & \leq & \P\left[\exists   k\le n+1:~M_{k} > \frac\eta{2} \sqrt{n}\right]  \cr
 & \leq & \P\left[\exists 1\leq  k\le n+1:~M_{k} > \frac\eta{2} \sqrt{n}\right] \cr 
 && + \P\left[M_{0} > \frac\eta{2} \sqrt{n}\right] \cr
 & \leq & n \P\left[ M_{1}>\frac\eta{2} \sqrt{n}\right] + \P\left[M_{0} > \frac\eta{2} \sqrt{n}\right]\cr
 & \xrightarrow[n\to\infty]{} & 0+0.
\end{eqnarray*} 
Here we used for the second summand  that $M_{0}$ is  almost surely finite and for the first summand we applied  
once again the existence of exponential moments and the Chebyshev inequality:  
$$n\P\left[ M_{1}> \frac{\eta\sqrt{n}}{2}\right] \leq C n \frac{\E[\exp(\delta M_{1})] }{\exp(\delta \sqrt{n})} \xrightarrow[n\to\infty]{}  0,$$ for some $\delta>0$.
The treatment of the second summand in (\ref{eq:thm:clt:final1}) is analogous by noting that
$$ \P\left[ v(n-R_{k(n)+1}) \geq \frac\eta{2} \sqrt{n}\right] \leq 
  \P\left[ \exists k\leq n:~ R_{k+1}-R_{k} \geq \frac\eta{2v} \sqrt{n}\right].$$
Altogether, the term in (\ref{eq:thm:clt:final1}) tends to $0$ for all $\eta>0$; this  finishes the proof  of  Theorem \ref{thm:CLT_planar}.

\subsection{Analyticity of $v$ and $\sigma^{2}$}
The aim of this section is to prove (real) analyticity of $v_{\mu}$ and $\sigma_{\mu}$. As any real function that is obtained as the restriction of a complex analytic function to a real neighborhood is a real analytic function we will in fact prove that $v_{\mu}$ and $\sigma_{\mu}$ extend to complex analytic functions. 
In this way we can use Generalized Vitali's Theorem that states that the uniform limit of a sequence of complex analytic functions on a fixed complex neighborhood is complex analytic.

Let   $\nu$ be a driving measure of a random walk with exponential moments. %, \emph{i.e.,}  $\E[\exp(\lambda d(X_{1},e))]<\infty$  for some $\lambda>0$. 
Furthermore, let $B$ be a finite subset of the support of $\nu$, \emph{i.e.,}  $B \subseteq supp(\nu)$.
Denote by $\Omega_{\nu} (B)$ the set of probability measures that give positive weight to all elements of $B$ and coincide with $\nu$ outside $B$. The set  $\Omega_{\nu} (B)$ can be identified with an open bounded convex subset in $\R^{|B|-1}$. We say, by abuse of notation, that $\O_{\mu}\subset \Omega_{\nu} (B) $ is an open neighborhood of $\mu\in \Omega_{\nu} (B)$ if its restriction to $B$ is an open neighborhood of the restriction of $\mu$ to $B$. We note $\O_{\mu}^{\C}$ for complex open neighborhoods of $\mu$ that may become smaller from line to line.

For each $\mu\in \Omega_{\nu}(B)$ we define the functions $v_{\mu}$ and $\sigma_{\mu}$ as the rate of escape and the asymptotic variance for the random walk with driving measure $\mu$, compare with Theorem \ref{thm:CLT_planar}. We write $\P^{\mu}$ (resp.~$\E^{\mu}$) for the  probability measure (resp.~expectation)  of the random walk corresponding to $\mu$  and $\E_{\bT}^{\mu}[\cdot]=\E^{\mu}_{x}[\cdot | D=\infty]$ with $T(x)=\bT$. 
 
\subsubsection*{Preparations} 
In order to show analyticity of $v_{\mu}$ and $\sigma_{\mu}$ we make the following preparations. Define  $D_{x}=\inf\{n\geq 1\,;\,Z_{n}\notin C_{\mathcal{A}}(x)\}$ and for $z\in C_{\mathcal{A}}(x)$ let $h^{\mu}(z)=h^{\mu}_{x}(z)=\P_z^{\mu}[D_{x}=\infty].$

In this section we consider surface groups with standard generating sets and assume $\nu$ to verify  Assumption \ref{ass:2}. Under this assumptions one verifies, as in the proof of Lemma \ref{lem:stayinconegen}, that  $h^\mu(z)>0$ for all $z\in C_{\A}(x)$.
\begin{lem}\label{lem:claim} Let $n\geq 1$, $x$ of type $\bT$ and $z_{0}, z_{1},\ldots, z_{n}\in C_{\A}(x)$. Then,
$$ \E^{\mu}_{z_{0}}[\1_{Z_{1}=z_{1}}\cdots \1_{Z_{n}=z_{n}} \mid D_{x}=\infty]=
 \E^{\mu}_{z_{0}}\left[\1_{Z_{1}=z_{1}} \cdots \1_{Z_{n}=z_{n}}\prod_{i=1}^{n} \frac{h^{\mu}(Z_{i})}{h^{\mu}(Z_{i-1})}\right].$$ 
\end{lem}

\begin{proof}
The proof of the  claim is a straightforward application of the Markov property; we just write it for $n=1$:
\begin{eqnarray*}
 \E^{\mu}_{z_{0}}[\1_{Z_{1}=z_{1}}\mid D_{x}=\infty] 
   &= & \frac{\P^{\mu}_{z_{0}}[Z_{1}=z_{1}, D_{x}=\infty]}{\P^{\mu}_{z_{0}}[D_{x}=\infty]}\cr
   & = &\P^{\mu}_{z_{0}}[Z_{1} = z_{1}]\frac{h^{\mu}(z_{1})}{h^{\mu}(z_{0})}
   =\E^{\mu}_{z_{0}}\left[\1_{Z_{1}=z_{1}} \frac{h^{\mu}(Z_{1})}{h^{\mu}(Z_{0})}\right].
\end{eqnarray*}
\end{proof}

\begin{lem}\label{lem:hanalytic} 
There exists some complex neighborhood $\O^{\C}_{\nu}(B)$ of $\Omega_{\nu} (B)$
such that for all $x\in \Gamma$ and all $z\in C_{\A}(x)$
the function $\mu\mapsto h^{\mu}(z)$  can be extended to a complex analytic function on $\O^{\C}_{\nu}(B)$. 
\end{lem}

\begin{proof}
 For each $ \mu\in \Omega_{\nu} (B)$ we have
$$1-h^{\mu}(z)=\P_{z}^{\mu}[D_{x}<\infty]=\sum_{k} \P_{z}^{\mu}[D_{x}=k].$$
Define $\gamma_{h}^{(k)}$ as the set of all paths $(z_{0},\ldots, z_{k})$ of length $k$ starting at $z$ and leaving 
$ C_{\A}(x)$ for the first time at time $k$, \emph{i.e.}, $z_{0}=z$, $z_{i}\in C_{\A}(x)~\forall 1\leq i \leq k-1,$ 
and $z_{k}\notin  C_{\A}(x)$. We write $x_{1},\ldots,x_{k}$ for the increments of a path in $\gamma^{(k)}_{h}$. The crucial observation now is that
$$\P_{z}^{\mu}[D_{x}=k]=
\sum_{(x_{1},\ldots x_{k})\in\gamma^{(k)}_{h}}\prod_{i=1}^{k} \mu (x_{j})$$ 
is a polynomial of degree at most $k$ and can therefore be naturally extended to a complex analytic function.

It follows from Equation (\ref{eq:PzD}) that $\P_{z}^{\mu}[D_{x}=k]\leq C_{D} c_{D}^{k}$ for all $z\in  C_{\A}(x)$ and some constants $C_{D}$ and $c_{D}<1$ that do not depend on $z$. Now, choose an open neighborhood  $\O_{\mu}^{\C}$ of $\mu$ such that there exists some $d\in[1,1/c_{D})$ such that
$$\max_{x\in B}\left\{\left|\frac{\tilde \mu(x)}{\mu(x)}\right|\right\}\leq d$$ 
for all $\tilde \mu\in \O_{\mu}^{\C}$. For each  $\tilde \mu\in \O_{\mu}^{\C}$ we have
\begin{eqnarray}\label{eq:lem:hanalytic:uniform}
|\P_{z}^{\tilde \mu}[D_{x}=k]|&=&|\E_{z}^{\mu}[\prod_{i=1}^{k} \frac{\tilde\mu (X_{i})}{\mu(X_{i})} \1_{D_{x}=k}]|=|\sum_{(x_{1},\ldots, x_{k})\in\gamma_{h}^{(k)}}\prod_{i=1}^{k} \frac{\tilde\mu (x_{i})}{\mu(x_{i})}\prod_{i=1}^{k} \mu(x_{i})|\cr
&\leq&  d^{k} \P_{z}^{\mu}[D_{x}=k]\leq C_{D} (dc_{D})^{k}.
\end{eqnarray}
Eventually, $1-h^{\tilde \mu}(z)$ is given locally as a uniform converging series of complex analytic functions on a fixed complex neighborhood that does not depend on the choices of $x$ and $z$. 
\end{proof}

\subsubsection*{Proof of Theorem \ref{thm:analytic}}
We have to prove that $\mu\mapsto v=\frac{\E^{\mu}[ d(Z_{ R_{2}}, Z_{ R_{1}})]}{\E^{\mu}[  R_{2}- R_{1}]}$ is real analytic. We will only prove that  the denumerator is an analytic function; the proof of the analyticity of the numerator is then a straightforward adaptation.  For the sake of simplicity we write $S_{k}, E_{k},$ and $D_{k}$ for $S_{k}^{(2)}, E_{k}^{(2)},$ and $D_{k}^{(2)}$. Moreover, we define $D_{0}=0$ and $E_{0}=E$. We have  
\[R_{2}-R_{1}= 1 + E\circ \theta^{R_{1}+1} + \sum_{k=1}^{\infty} (S_{k}-S_{k-1}) \1_{S_{1}<\infty,\ldots, S_{k}<\infty}.\]
Therefore, the denumerator can be written as 
\begin{equation}\label{eq:vfrac}
 \E^{\mu}[  R_{2}- R_{1}]=\E_{\bT}^{\mu}[1 + E\circ \theta^{1}]+ \sum_{k=1}^{\infty} \E_{\bT}^{\mu}[(S_{k}-S_{k-1}) \1_{S_{1}<\infty,\ldots, S_{k}<\infty}].
 \end{equation} 
 Since $S_{k}-S_{k-1}=D_{k-1}+E_{k-1}$ we first prove that $\E_{z}^{\mu}[E\mid D_{x}=\infty]$ and $\E_{z}^{\mu}[D\mid D_{x}=\infty]$ can be extended to complex analytic functions.

Due to Lemmata \ref{lem:stayinconegen} and \ref{lem:hanalytic} we can choose  an open neighborhood 
$\O^{\C}_{\mu}$ of $\mu$ and some $c_{h}>0$ such that $|h^{\tilde \mu}(z)|\geq c_{h} $ for all $\tilde\mu\in\O_{\mu}^{\C}$ 
and all $z\in  C_{\A}(x)$.
Recall, that $E$ has exponential moments, \emph{i.e.}, there exist constants $C_{E}$ and $c_{E}<1$ such that  $\P_{z}^{\mu}[E=k]\leq C_{E} c_{E}^{k}$ for all $z\in C_{\A}(x)$.
Now, we choose  $\O_{\mu}^{\C}$ smaller, if necessary, in order to guarantee that, for all $\tilde{\mu}\in \O_{\mu}^{\C}$,
$$\max_{x\in B}\left\{\left|\frac{\tilde \mu(x)}{\mu(x)}\right|\right\}\leq d<1/c_{E}.$$ For each $\mu$ we have
$$ \E_{z}^{ \mu}[E\mid D_{x}=\infty]=\sum_{k=0}^{\infty} k \E_{z}^{\mu}[ \1_{E=k}\mid D=\infty].$$
Consider
\begin{equation*}
f_{k}(\mu)= k \E_{z}^{\mu}[\1_{E=k}\mid D=\infty]  =   \E_{z}^{\mu}[k \prod_{j=1}^{k}\frac{h^{\mu}(Z_{j})}{h^{\mu}(Z_{j-1})} \1_{E=k} ] 
 =  \sum_{(z_{1},\ldots z_{k})\in\gamma^{(k)}_{E}}k \frac{h^{ \mu}(z_{k})}{h^{\mu}(z)}\prod_{i=1}^{k} {\mu (x_{i})} ,
\end{equation*} 
where $\gamma^{(k)}_{E}$ is the set of all paths of length $k$ corresponding to the event $\{E=k\}$.
Hence, each $f_{k}({\mu})$ can be extended to a complex  analytic function on $\O_{\mu}^{\C}$. If necessary we have to choose $\O_{\mu}^{\C}$ still smaller such that $h^{\mu}(z)$ is analytic on $\O_{\mu}^{\C}$. 

Now, for every $\tilde \mu\in \O_{\mu}^{\C}$:
$$|f_{k}(\tilde \mu)|\leq C k d^{k} c_{E}^{k} \frac1{1-c_{h}}$$
and hence $\E_{z}^{ \mu}[E\mid D_{x}=\infty]$ is given locally as a uniform converging series of complex analytic functions on a fixed neighborhood $\O_{\mu}^{\C}$ and therefore is analytic on $\O_{\mu}^{\C}.$  The proof of the analyticity of $\E_{z}^{ \mu}[D\mid D_{x}=\infty]$ is  similar and therefore omitted.

Let us return to the denumerator, see  Equation (\ref{eq:vfrac}), that can we written as 
$$\E^{\mu}[  R_{2}- R_{1}]=\sum_{k=0}^{\infty} g_{k}(\mu)$$
where
$$g_{0}(\mu)=\E_{\bT}^{\mu}[1 + E\circ \theta^{1}]\mbox{ and }  g_{k}(\mu)=\E_{\bT}^{\mu}[S_{k}-S_{k-1} \1_{S_{1}<\infty,\ldots, S_{k}<\infty}]\quad k\geq1.$$ 

The same arguments as above imply that the functions
$\E_{z}^{\mu}[D+E| D_{x}=\infty],~\P_{z}^{\mu}[S_{1}<\infty,\cdots, S_{k-2}<\infty| D_{x}=\infty]$ and $\E_{z}^{\mu}[D_{k-1}+E_{k-1} \1_{D_{k-1}<\infty}| D_{x}=\infty]$ extend to complex analytic function on some neighborhood $\O_{\mu}^{\C}$. The choice of the neighborhood is again independent of the choices of $x$, $z$ and $k$. In the same way we see that $g_{0}(\mu)$ extends to a complex analytic function. Furthermore, for $k\geq 1$ and $x$ of type $\bT$,
\begin{eqnarray*}
g_{k}(\mu) &=&  \E_{\bT}^{\mu}[S_{k}-S_{k-1} \1_{S_{1}<\infty,\ldots, S_{k}<\infty}]\cr
& =& \E_{x}^{\mu} [D_{k-1}+E_{k-1} \1_{D_{k-1}<\infty}| D_{x}=\infty] \P^{\mu}_{x}[S_{1}<\infty,\cdots, S_{k-2}<\infty| D_{x}=\infty],
\end{eqnarray*}
and hence  $g_{k}(\mu)$ extends to a complex analytic function on some neighborhood $\O_{\mu}^{\C}$.
In order to prove the analyticity of $\E^{\mu}[  R_{2}- R_{1}]$ it suffices to show that $\sum_{k} |g_{k}(\mu)|$ converges uniformly in some open neighborhood of $\mu$. To this end we use again the exponential moments of the involved random variables and Lemmata \ref{lem:stayinconegen} and \ref{lem:hanalytic} to obtain that there exists a complex neighborhood $\O_{\mu}^{\C}$ such that for all $\tilde\mu \in \O_{\mu}^{\C}$ we have  that 
$$
|\E_{x}^{\tilde \mu} [D_{k-1}+E_{k-1} \1_{D_{k-1}<\infty}| D_{x}=\infty]|\leq C \mbox{ and }  |\P^{\tilde \mu}_{x}[S_{1}<\infty,\cdots, S_{k-2}<\infty| D_{x}=\infty]|\leq (1-c)^{k-2}
$$
and eventually that $|g_{k}(\tilde \mu)|\leq C (1-c)^{k-2}$ for all $\tilde\mu \in \O_{\mu}^{\C}$.
This finishes the proof of the analyticity of the denumerator. 
As mentioned above the proof of the analyticity of the numerator is very similar and therefore omitted. Since $\E^{\mu}[  R_{2}- R_{1}]>0$ this proves the analyticity of the rate of escape.  The proof of the analyticity of the asymptotic variance is a straightforward adaption of the proof above.

\subsection{Bypassing of Assumption \ref{ass:2}}\label{sec:bypass}
Let $\Gamma$ be a surface group and $\mu$ be a probability measure on $\Gamma$ with a finite exponential moment and  whose support
generates $\Gamma$ as a semigroup.
Due to the irreducibility there exists some $\ell\in \N$ such that
$$\bar\mu:=\frac1\ell\sum_{i=1}^{\ell} \mu^{(\ell)}$$ 
fulfills Assumption \ref{ass:2}. Let $(\bar X_{j})_{j\geq 1}$ be 
a sequence of i.i.d.~random variables with law $\bar\mu$ and 
$\bar Z_{n}=\prod_{j=1}^{n}\bar X_{j}$ the corresponding random walk. 
Due to its construction the variables $\bar X_{j}$ can be seen as 
the result of  a \emph{two-step probability event}: 
let $(U_{j})_{j\geq 1}$ be i.i.d.~random variables 
(independent of $(X_{i})_{i\geq 1}$ with uniform distribution on 
$\{1,\ldots,\ell\}$ then
$\bar X_{j}\stackrel{\mathcal{D}}{=} Z_{U_{j}}$ for all $j\geq 1$ on an appropriate joint probability space. Define $T_{n}=\sum_{j=1}^{n} U_{j}$ then 
$$\bar Z_{n}\stackrel{\mathcal{D}}{=} Z_{T_{n}}.$$
The proof of Theorem \ref{thm:LLN_planar} can now be adapted as follows. 
Denote $\bar R_{n}$ the renewal times corresponding to $\bar Z_{n}$. 
Define
$k(n)=\max\{k: T_{\bar R_{k}<n}\}$. As $T_{n}/n\to (\ell+1)/2$ a.s.~and $\bar R_{k}/k \to \E[\bar R_{2}-\bar R_{1}]$ a.s.~we have that 
$$\frac{k(n)}n \xrightarrow[n\to\infty]{a.s.}  \frac{\ell+1}2 \E[\bar R_{2}-\bar R_{1}].$$ 
Eventually, we obtain 
\begin{equation}\label{eq:v:general} 
\frac1n d(Z_{n},e) \xrightarrow[n\to\infty]{a.s.} 
\frac{\ell+1}2  \frac{\E[ d(\bar Z_{ \bar R_{2}}, \bar Z_{ \bar R_{1}}])}{\E[  \bar R_{2}- \bar R_{1}]}.
\end{equation}
In the same spirit one can adjust the proof of Theorem \ref{thm:CLT_planar} and 
obtains a formula for the asymptotic variance

\begin{equation}
\sigma^{2}=\frac{2 \E[(d(\bar Z_{\bar R_{2}}, \bar Z_{\bar R_{1}}) - (\bar R_{2}- \bar R_{1})v)^{2}]}{(\ell+1)
\E[\bar R_{2}-\bar R_{1}]}.
\end{equation}
Due to the above formul\ae~for the rate of escape and asymptotic variance the results on analyticity 
also hold without Assumption \ref{ass:2}.

\section{Discussion}\label{sec:discussion}

The key ingredient that we used for the renewal theory is that for all $x\in \Gamma$
\begin{equation}
\P_{x}[Z_{n}\in C(x) \mbox{ for almost all } n]>0.
\end{equation}
This fact does not hold in general as  Example \ref{ex:counter1} shows.

\begin{ex}\label{ex:counter1}
Let $\Gamma$ be a hyperbolic group with generating set $S$ and neutral element $e$. We set $\Gamma'=\Gamma\times(\Z/2\Z)$ with generating set $S'=\{(s,0),~s\in S\}\cup \{(e,1)\}$. So the Cayley graph of $\Gamma'$ consists of two copies of the Cayley graph of $\Gamma$ that are connected by edges between $(x,0)$ and $(x,1)$ for all $x\in \Gamma$. Observe that every geodesic starting from the origin $(e,0)$ that goes through a point $(x,1)$ will never visit the $0$-level afterwards. Eventually, while the cones types of the level $1$ may be ubiquitous they are not large. In particular we have for all $x$ on level $1$ that $\P_{x}[Z_{n}\in C(x) \mbox{ for almost all } n]=0.$ 
\end{ex}

This fact was the motivation of the definition of  large cone types. Indeed, the existence of large cone types is almost necessary to the renewal structure we have defined. Let $\Gamma$ be a non-amenable finitely generated group endowed with a word metric
and a probability measure $\mu$ whose support generates $\Gamma$ as a semigroup. Let $A\subset\Gamma$ and let us consider $D_A=\inf\{n\geq 0:\ Z_n\notin A\}$. If $\P_x[D_A=\infty]>0$ then it is straightforward to show that 
$$\P_x\left[ \lim_{n\to\infty} d(Z_n,^c\!A)=\infty| D_A=\infty\right]=1\,.$$

It is possible to replace the cones in the renewal structure by quasi-cones. By this we mean that we replace the geodesics  in the definition of the cones by quasi-geodesics. One way to establish such a construction is to consider shadows of finite balls or symbolic dynamics as described in  \cite{CP:93}.  The quasi-cones obtained from such a construction are large in our sense. While such a construction might be of its own interest it seems not to imply  a central limit theorem.

Let us therefore end with two questions.
\begin{que}
Does any hyperbolic (automatic) group have large cone types  for some (all) finite  generating set ?
\end{que}
\begin{que}
Does any hyperbolic (automatic) group have ubiquitous cone types  for some (all) finite generating set ?
\end{que}

\section*{Acknowledgment}
The authors wish to express here their deep gratitude towards F.~Ledrappier for his unswerving support and to the anonymous referee for his/her comments.

\begin{small}
\addcontentsline{toc}{chapter}{Bibliography}
\bibliography{bib}
\end{small}

\begin{tabular}{lll}
Peter Ha\"{\i}ssinsky &\hspace{1cm} & Pierre Mathieu and Sebastian M\"uller \\
IMT & &I2M\\
Université Paul Sabatier & &Aix-Marseille Universit\'e\\
118 route de Narbonne & & 39, rue F. Joliot Curie\\
31062 Toulouse Cedex 9& &13453 Marseille Cedex 13\\
France & &France\\
%&{\tt mueller@cmi.univ-mrs.fr}\\
\end{tabular}

\end{document}